\documentclass[11pt,a4paper]{article}

\usepackage{graphicx, graphics, epsfig, color}
\usepackage{booktabs}
\usepackage{subcaption}
\usepackage{natbib}
\usepackage{amsmath,amssymb,amsthm}
\usepackage{tikz, pgfplots}
\usetikzlibrary{plotmarks,spy}
\pgfplotsset{compat=newest}
\usepackage[margin=1in]{geometry}
\usepackage{hyperref}
\hypersetup{colorlinks=false}
\usepackage[linesnumbered,ruled,vlined]{algorithm2e}
\usepackage{enumitem}
\usepackage{cleveref}

\definecolor{RED}{rgb}{0.7,0,0}
\definecolor{BLUE}{rgb}{0,0,0.69}
\definecolor{GREEN}{rgb}{0,0.6,0}
\definecolor{PURPLE}{rgb}{0.69,0,0.8}

\usepackage{booktabs}  
\usepackage{array}     
\usepackage{multirow}
\usepackage{amsthm}

\usepackage[T1]{fontenc}
\usepackage[utf8]{inputenc}

\newtheorem{theorem}{Theorem}   
\newtheorem{lemma}{Lemma}       
\newtheorem{proposition}{Proposition} 
\theoremstyle{definition}
\usepackage{mathrsfs}
\title{Adaptive control for multi-scale stochastic dynamical systems with stochastic next generation reservoir computing}

 \author{
 \normalsize{
Jiani Cheng$^{1,2,}$\footnotemark[2]\ ,
Ting Gao$^{1,2,3}$\footnotemark[1]\ ,
Jinqiao Duan$^{4,5,}$\footnotemark[3]
}\\[10pt]
\footnotesize{$^1$ School of Mathematics and Statistics, Huazhong University of Science and Technology, Wuhan, China} \\
\footnotesize{$^2$ Center for Mathematical Science, Huazhong University of Science and Technology, Wuhan, China} \\
\footnotesize{$^3$ Steklov-Wuhan Institute for Mathematical Exploration, Huazhong University of Science and Technology, China} \\
\footnotesize{$^4$ Department of Mathematics and Department of Physics, Great Bay University, Dongguan, China} \\
\footnotesize{$^5$ Guangdong Provincial Key Laboratory of Mathematical and Neural Dynamical Systems, Dongguan, China. }
}

\begin{document}
\maketitle
\begin{abstract}
The rapid advancement of neuroscience and machine learning has established data-driven stochastic dynamical system modeling as a powerful tool for understanding and controlling high-dimensional, spatio-temporal processes. We introduce the stochastic next-generation reservoir computing (NG-RC) controller, a framework that integrates the computational efficiency of NG-RC with stochastic analysis to enable robust event-triggered control in multiscale stochastic systems. The asymptotic stability of the controller is rigorously proven via an extended stochastic LaSalle theorem, providing theoretical guarantees for amplitude regulation in nonlinear stochastic dynamics. Numerical experiments on a stochastic Van-der-Pol system subject to both additive and multiplicative noise validate the algorithm, demonstrating its convergence rate across varying temporal scales and noise intensities. To bridge theoretical insights with real-world applications, we deploy the controller to modulate pathological dynamics reconstructed from epileptic EEG data. This work advances a theoretically guaranteed scalable framework for adaptive control of stochastic systems, with broad potential for data-driven decision making in engineering, neuroscience, and beyond.\\


\end{abstract}

\section{Introduction}

The study of nonlinear oscillatory systems has broad applications in physics, biomedicine, and engineering\cite{Hayashi+1986, Mohazzabi2004TheoryAE, Rafiq2024CriticalAF}. Within biological and engineered systems, oscillatory behaviors often emerge due to intrinsic feedback mechanisms and external perturbations\cite{ashall2009pulsatile, gerard2009temporal, goldbeter2012systems}. Controlling such oscillatory dynamics is crucial in many domains, including stabilizing neural rhythms, regulating gene expression, and designing robust mechanical systems. Classical approaches, including bifurcation analysis\cite{Ajjarapu1991BifurcationTA} and Lyapunov stability theory, provide qualitative insights but face computational bottlenecks in high-dimensional stochastic systems. Besides, oscillatory dynamics often exhibits complex behaviors such as limit cycles\cite{Manchester2010TransverseDA}, quasiperiodicity\cite{Zanotti2002QuasiperiodicAA}, and chaos\cite{Belhaq1999QuasiPeriodicOC}, making traditional control methods difficult to apply.

Real-world systems are inherently stochastic, where noise plays a fundamental role in shaping their dynamical behavior\cite{xu2024data}. Thus, the challenge of designing controllers for stochastic dynamical systems has become a central problem, particularly in contexts where deterministic models fail to capture essential characteristics of the system. Stochastic control methods, including those based on stochastic stability theory\cite{Zhu2020StabilityAF} and convex optimization\cite{Bai2020safety}, have been developed to address these problems. Recently, the Schrödinger bridge has also been investigated to control dynamics between meta-stable invariant sets\cite{Zhang2024ActionFA}. In addition, many oscillatory systems exhibit multiple time scales, where slow and fast dynamics interact non-trivially. Although machine learning methods such as neural SDEs\cite{kidger2021neural, kidger2021efficient} and reservoir computing (RC)\cite{tanaka2019recent, canaday2021model} have been applied to stochastic modeling, their potential in adaptive control remains underexplored. In particular, next-generation reservoir computing (NG-RC)\cite{gauthier2021next} has demonstrated superior efficiency in learning dynamical systems, but has seen limited applications in SDE control. This gap highlights the need for an adaptive control framework capable of effectively handling noise and multi-scale dynamics.

In this work, we propose a novel noise-adapted stochastic next-generation reservoir computing (S-NGRC) controller to address these challenges. Our approach integrates NG-RC with stochastic control theory, providing a data-driven framework for trajectory tracking in stochastic dynamical systems. The stochastic NG-RC controller leverages a closed-loop feedback mechanism to infer and adjust the control input, ensuring precise convergence to target trajectories. Theoretical guarantees on control error convergence are established through the stochastic LaSalle theorem, offering a rigorous foundation for adaptive control in stochastic oscillatory systems. Furthermore, hardware implementations of NG-RC\cite{kent2024controlling, kent2024controlling2} have demonstrated its efficiency in chaotic system control, suggesting its potential for real-time applications. Our framework advances the state-of-the-art by bridging machine learning-based control with stochastic stability analysis, offering a promising direction for applications in complex biomedical and engineering systems.

The remainder of this paper is structured as follows. In Section 2, we review related work on data-driven control methods for complex systems. In Section 3, we present the design of the stochastic NG-RC controller, formulated under the stochastic differential equation framework, and establish the coupling mechanism between the state space of the controlled system and the reservoir. In Section 4, we analyze the control stability by deriving a stochastic LaSalle theorem based on the Lyapunov stability theory and prove the asymptotic convergence of the control error. And in Section 5, experimental results are reported, including robustness verification on stochastic Van-der-Pol systems under different noise types, intensities, and time scales, as well as an empirical study on seizure suppression using EEG data. Finally, in Section 6, we summarize the key contributions of this work in data-driven control theory, stability analysis of stochastic controlled systems, and biomedical applications, discuss current limitations, and outline directions for future research.

\section{Related work}

In recent years, the core issue of frequency modulation (FM) and amplitude modulation (AM) as the heart of oscillatory system control has been continuously deepened in research in biology, physics, and engineering. During learning processes, the hippocampus and the inferior temporal cortex of the brain exhibit a significant improvement in the amplitude of theta rhythms\cite{colgin2013mechanisms} while maintaining stable frequencies, which is a typical AM phenomenon. For instance, intermittent theta burst stimulation (iTBS) in the cerebellum can improve walking function in stroke patients, and its mechanism may be related to the amplitude regulation of corticospinal excitability\cite{di2008physiological}. Ge et al. (2014)\cite{ge2014achieving} based on the Andronov-Hopf bifurcation theory, mapped the system to a complex space through linear feedback, explicitly correlating parameters with frequency/amplitude relationships, but limited to a single regulatory target. Qin et al. (2021)\cite{qin2021frequency} pointed out that amplitude regulation depends on the transfer of energy between different modes (such as the redistribution of ionic channel energy in neuronal discharge), while frequency regulation is related to the timing characteristics of signal transmission. This discovery explains the phenomenon that AM/FM may be regulated through independent pathways in biological systems. Although experiments have confirmed that AM/FM can be independently regulated, it remains unclear whether they share underlying mechanisms (such as feedback loop structures). For example, the synergistic action of positive and negative feedback in the nervous system may dominate FM, while AM is more dependent on a single feedback or external input\cite{qin2021frequency}. Moreover, the impact of noise in biological systems (such as the random firing of neural activity) on AM/FM regulation has not yet been fully studied. It is necessary to integrate stochastic control theory (such as the Fokker-Planck equation) to enhance the robustness of controllers.

Data-driven non-linear system modeling and control methods have attracted considerable attention in the study of complex dynamic systems. Among them, Sparse Identification of Nonlinear Dynamics (SINDy) utilizes innovative sparse regression techniques to select key nonlinear terms from a high-dimensional library of candidate functions, constructing a concise explicit dynamical equation. This approach exhibits unique advantages in terms of model interpretability and computational efficiency\cite{Papageorgiou2024SlidingmodeCO}. The framework has been successfully extended into the field of control: Brunton et al. (2016)\cite{Brunton2016SparseIO} pioneered the SINDy control theory framework, and subsequent research has further integrated it with Model Predictive Control (MPC), achieving robust tracking of complex systems through online optimization\cite{Zhang2018OnTC,Wang2023AnIK}. This has demonstrated a significant advantage of reducing computational time by 30$\%$ compared to traditional methods in robotic arm control\cite{Guevara2024PDBasedAS}.

Similarly to SINDy's need for explicitly constructing a library of basis functions, the next-generation reservoir computing (NG-RC) builds high-order polynomial basis functions as reservoir states and directly learns the discrete state transition mapping using ridge regression. However, there is an essential difference in methodology: SINDy focuses on identifying dynamical models in the form of differential equations from continuous-time data, with the sparse optimization process explicitly eliminating redundant terms through algorithms like LASSO\cite{Brunton2015DiscoveringGE}. In contrast, NG-RC directly characterizes system evolution through discrete state transition mapping and implicitly compresses the weights of basis functions using ridge regression. In terms of real-time performance, NG-RC, with its single-step inference property, is more suitable for embedded control scenarios, while the combination of SINDy and MPC requires an online solution of optimization problems. Notably, NG-RC\cite{gauthier2021next} introduces historical states to enhance the representation of temporal dynamics, echoing the phase-space reconstruction of Takens' embedding theorem\cite{noakes1991takens}. SINDy, on the other hand, is more dependent on prior knowledge of the physical properties of the system through the candidate function library\cite{Rudy2016DatadrivenDO}. These two methods also present complementary features in terms of noise robustness. SINDy is sensitive to data noise and requires preprocessing, while NG-RC naturally suppresses noise interference through regularization mechanisms. This difference makes SINDy more advantageous in scenarios that require precise dynamical analysis, such as soft robot control\cite{Papageorgiou2024SlidingmodeCO}, while NG-RC performs better in real-time control tasks with a high demand for rapid response.

\section{Controlling stochastic dynamical systems via S-NGRC }

A typical stochastic differential equation(SDE) can be expressed as follows:
\begin{equation*}
d X_t = f\left(X_t\right) d t + g\left(X_t\right) d W_t, \quad X_0 = x_0,
\end{equation*}
where the $N$-dimensional vector function $f(\cdot)$ represents the drift term, the $N \times N$ matrix function $g(\cdot)$ represents the diffusion term, and $W_t$ denotes an $N$-dimensional Brownian motion.
Naturally, we need to exert control over both the deterministic and stochastic components at the same time.
\begin{equation}\label{sde1}
d X_t = [f\left(X_t\right)+u_1(X_t)] d t + [g\left(X_t\right)+u_2(X_t)] d W_t, \quad X_0 = x_0.
\end{equation}
Here $u_1$ and $u_2$ are non-linear feedback controls related to the state $x(t)$ of the system learned by the data-driven approach. In the following, we present our method in two steps.

\subsection{Step 1: learning data dynamics via stochastic NG-RC}
\label{sec:model}

Consider a simplified nonlinear stochastic differential equation (SDE)
\begin{equation}\label{sde2} 
d X_t = \left( f\left(X_t\right)+u_t \right) d t + \sigma d W_t, \quad X_0 = x_0, 
\end{equation}
where $\sigma$ is a constant matrix and $u_t = u(X_t)$ represents the controller. Here, we simplify the diffusion term in \eqref{sde1}, retaining only the effect of the external input $u_t$ on the drift term. The time interval $[0, T]$ is discretized into equally spaced subintervals $K$ with a step size of $\Delta t = T / K$, where discrete time instants are denoted as $t_i = i \Delta t$ for $i = 0,1,\dots,K$. Using the Euler-Maruyama discretization scheme, \eqref{sde2} is approximated as  
\begin{equation}\label{sde3}
X_{i+1} = X_i + \left[f\left(X_i\right) + u_i \right] \Delta t + \sigma \sqrt{\Delta t} \xi_i,
\end{equation}
where $X_i \in \mathbb{R}^n$ is the system state at time $t_i$. $u_i \in \mathbb{R}^n$ is the control input, and $\xi_i \sim \mathcal{N}\left(0, I_N\right)$ is an independent and identically distributed (i.i.d.) standard Gaussian noise term. To ensure numerical precision, the size of the  the discretization step must satisfy the constraint $\Delta t < \frac{1}{\|\nabla f(X)\|_{\max}}$, where $\|\nabla f(X)\|_{\max}$ represents the maximum spectral norm of the Jacobian matrix of the drift term.  

\begin{figure}[htbp]
    \centering
    \includegraphics[width=\textwidth]{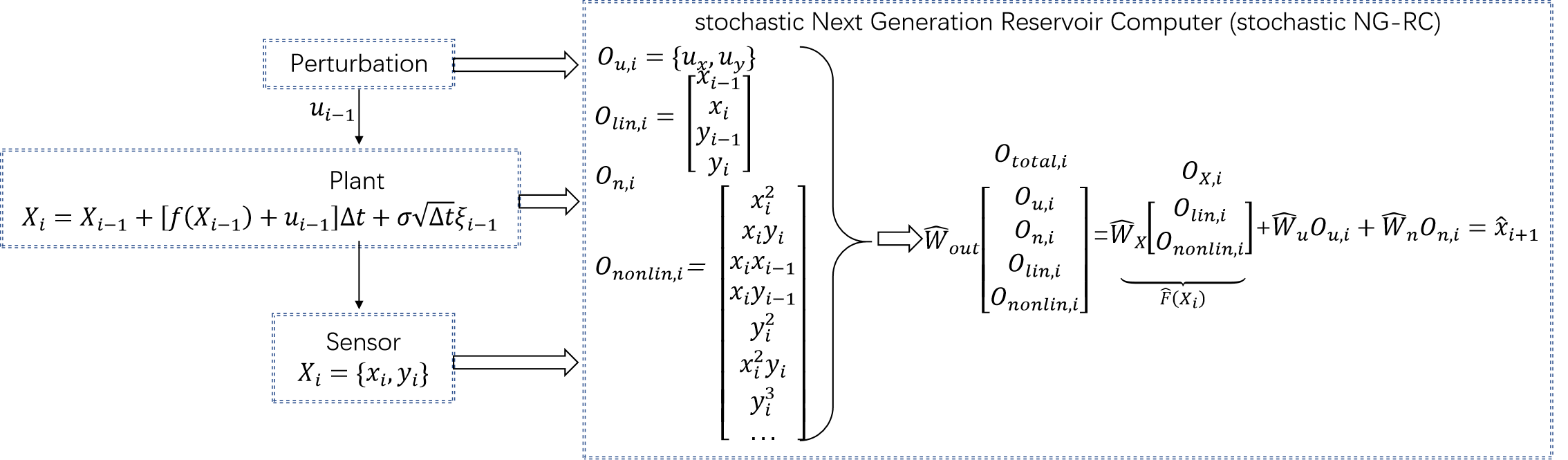}
    \caption{Stochastic NG-RC framework diagram (take 2D sde as an example).}
    \label{fig:NG-RC}
\end{figure}

Following the nonlinear control framework for discrete-time systems developed in \cite{sarangapani2018neural}, we abstract system \eqref{sde3} as  
\begin{equation}\label{ngrc1}
\mathcal{X}_{i+1} = \mathcal{F}\left(\mathcal{X}_i, u_i, \sigma \right),
\end{equation}
where $\mathcal{F}$ is a mapping that describes the controlled $N$-dimensional system. In the case of full observation, the system output $Y_t$ is equivalent to the system state $X_t$. Based on the structure of the drift term in the SDE, we construct a feature vector $O_{X,i}$ that captures the system state by including monomials of different orders. The feature vector consists of a linear term, $O_{lin,i} = \{X_i, X_{i-1}\}$ \cite{kent2024controlling2}, and a nonlinear term $O_{nonlin,i}$ (including cross-products up to the third order). In this study, we only consider the effect of the previous and current state on the next state. However, as shown in \cite{gauthier2021next}, delay embeddings spanning multiple time steps can also be considered. The external input term $u_t$ and the noise term $\sigma \sqrt{\Delta t} \xi$ are incorporated into the feature vectors, denoted as $O_u$ and $O_n(\Omega)$, respectively.  

Based on Eq. \eqref{ngrc1}, the structure of the stochastic NG-RC model (illustrated in Fig. \ref{fig:NG-RC}) is given by  
\begin{equation}\label{ngrc2}
\mathcal{X}_{i+1} = W_X O_{X,i} + W_u O_{u,i} + W_n O_{n,i},
\end{equation} 
where $O_{u,i} = u_i \in \mathbb{R}^n$, and $W_X, W_u, W_n$ are the weight matrices associated with the three feature vectors.  

During the training phase of the NG-RC controller, the external input $u_t$ is a predesigned perturbation, allowing NG-RC to learn how control signals influence the state of the system in \eqref{sde2}. The feature vectors on the right-hand side of Eq. \eqref{ngrc2} can be combined into a single vector, $O_{total,i} = \{O_{X,i}, O_{u,i}, O_{n,i}\}$ (as illustrated in Fig. \ref{fig:NG-RC}), leading to the compact formulation  
\begin{equation}
\mathcal{X}_{i+1} = W_{out} O_{total}.
\end{equation}  
Similar to traditional reservoir computing, the output weights of the stochastic NG-RC model, $W_{out}$, are obtained via Tikhonov regularization
\begin{equation*}
W_{out} = X \mathbb{O}_{total}^{\mathrm{T}}\left(\mathrm{O}_{total} \mathrm{O}_{total}^{\mathrm{T}}+\alpha \mathbf{I}\right)^{-1},
\end{equation*}
where $X = [X_0, X_1, \dots, X_K]$, and $\alpha$ is a ridge regression parameter that regulates the model’s generalization performance. At this stage, the noise feature vector $O_n(\Omega)$ serves the same role as the random term in the Euler-Maruyama discretization process.  

A key distinction between the proposed stochastic NG-RC model and the traditional NG-RC framework is the explicit incorporation of the noise feature vector $O_n$. When addressing the control of SDEs, one natural consideration is controlling the diffusion term. In the simplified system \eqref{sde2}, we do not explicitly derive the mathematical structure of the stochastic NG-RC under the influence of $u_2$ in system \eqref{sde1}. However, numerical experiments demonstrate that the proposed model remains robust even in the presence of multiplicative noise.

\subsection{Step 2: designing control law for stochastic dynamical systems}

After training the NG-RC controller, we define the tracking error at time step $i+1$ as  
\begin{equation*}
e_{i+1} = \hat{X}_{i+1} - X_{des,i+1},
\end{equation*}
where $X_{des}$ denotes the desired trajectory, and $\hat{X}_{i+1}$ represents the observed system state after applying the control input $u_i$ generated by the stochastic NG-RC controller at time step $i$. Assuming that $\hat{X}_{i+1}$ follows the dynamical behavior governed by Eq. \eqref{ngrc2}, we substitute Eq. \eqref{ngrc2} into the above expression, yielding  
\begin{equation}\label{e}
    e_{i+1} = W_X O_{X,i} + W_u O_{u,i} + W_n O_{n,i} - X_{des,i+1}.
\end{equation}
To achieve the desired control, assuming 
\begin{equation*}
    e_{i+1} = K e_i,
\end{equation*}
we design the control signal as  
\begin{equation}\label{u}
O_{u,i} = W_u^{-1} \left[ X_{des,i+1} - W_X O_{X, i} - W_n O_{n,i} + K e_i \right].
\end{equation}
To ensure error convergence, the spectral radius of the closed-loop gain matrix $K$ must satisfy $\rho(K) < 1$. However, in dynamical systems, the control error $e_t$ should also exhibit appropriate dynamic behavior. In the following, we analyze the nonlinear controlled stochastic differential equation with respect to its calibration and error convergence conditions.

\section{Stablity analysis}
Lyapunov stability theory is a fundamental tool for analyzing the stability of dynamical systems, particularly in the study of nonlinear systems. While classical Lyapunov stability conditions are typically applied to equilibrium points, stochastic differential equations (SDEs) often exhibit more complex invariant structures, such as limit cycles and quasi-periodic orbits. To address this, we employ LaSalle’s theorem \cite{Lasalle1976TheSO}, an extension of Lyapunov’s method, which allows stability analysis through the study of invariant sets.

In this work, we extend the stochastic LaSalle theorem to controlled SDEs, establishing sufficient conditions for trajectory convergence under noise. Our results provide a theoretical foundation for the proposed stochastic NG-RC controller, ensuring that the controlled system asymptotically approaches the desired trajectory. The detailed mathematical derivations and proofs are presented in Supplementary material~\ref{Supplementary material:stability}.

\section{Experiments and results}



\subsection{Adaptive control of multi-scale stochastic dynamical systems}

Event-Triggered Control (ETC) is an intelligent control strategy whose core idea is to dynamically decide the moment of control update based on the change of system state or error signals, instead of executing control calculations according to a fixed time interval.Compared with the traditional Time-Triggered Control (TTC), ETC can reduce the number of control updates under the premise of guaranteeing the system performance, thus saving computational resources and communication bandwidth, and improving the real-time and robustness of the control system. The main categories of ETC include Static ETC and Dynamic ETC \cite{tang2022periodic}. Static ETC sets the trigger conditions through a fixed error threshold, such as the trigger strategy based on Lyapunov stability analysis, with typical rules such as $\|e(t)\| \geq \sigma$, and triggers the control update when the error exceeds the preset threshold $\sigma$. Dynamic ETC \cite{yi2018dynamic}, on the other hand, allows adaptive adjustment of the state value over time or the state of the system, such as self-triggered control (STC) and predictive ETC, where STC anticipates future trigger moments and thus determines the next trigger at the time of the last control execution, and predictive ETC combines model prediction methods to optimize the timing of control updates.

In complex nonlinear systems such as neural dynamics or oscillatory models, certain undesired states—like epileptic seizures—can be viewed as metastable attractors that the system may drift toward under external perturbations. To prevent such transitions, we incorporate insights from bifurcation theory and event-triggered control (ETC) to construct a control framework that actively monitors and regulates the system near critical thresholds. Instead of continuously applying control inputs at fixed intervals, ETC triggers interventions only when the system state approaches a bifurcation boundary or the tracking error exceeds a predefined margin. This selective activation enables both stability assurance and computational efficiency in managing critical transition risks.Key components of this approach include modeling the underlying bifurcation structures, designing hybrid triggering conditions, performing Lyapunov-based stability analysis, and dynamically adjusting the safety margin based on the system's evolving behavior.

To verify the effectiveness of the proposed data-driven controller with stochastic NG-RC, we use a multi-timescale stochastic Van der Pol oscillator\eqref{VDPadd} as a benchmark. To further study the interactive effects of multi-timescales and noise intensity on the performance of the controller, experiments are carried out under both the additive noise and multiplicative noise. Specifically, for additive noise, we divide the analysis into three subcategories: (1) small noise without multiple time scales, (2) comparison with the traditional NG-RC model, and (3) robustness analysis under varying noise intensity and multiple-time-scale effects.


\subsubsection{Case 1: Multi-scale SDEs under additive noise}
\label{sec:experiment}

Consider the following stochastic Van der Pol system with multiple time scales
\begin{equation}\label{VDPadd}
\left\{\begin{array}{l}d x=-y d t+\sigma_1 d W_t^1, \\ d y=\frac{1}{\varepsilon}\left(y-\frac{1}{3} y^3+x\right) d t+\frac{\sigma_2}{\sqrt{\varepsilon}} d W_t^2.\end{array}\right. 
\end{equation}
where $\varepsilon\leq1$ is a constant representing the time scale, $x_t \in \mathbb{R}$ is the slow variable, and $y_t \in \mathbb{R}$ is the fast variable. From equation \eqref{sde2}, the corresponding controlled SDE can be obtained
\begin{equation}\label{cVDPadd}
\left\{\begin{array}{l}d x=[-y+u_1] d t+\sigma_1 d W_t^1, \\ d y=[\frac{1}{\varepsilon}\left(y-\frac{1}{3} y^3+x\right)+u_2] d t+\frac{\sigma_2}{\sqrt{\varepsilon}} d W_t^2.\end{array}\right. 
\end{equation}

\paragraph{Fast modulation in low-intensity noise without multiple time scales:}
In order to clearly demonstrate the automatic identification of the controller and fast response to triggered event, the noise intensity $\sigma_1,\sigma_2 = 0.1$ and the time scale $\epsilon = 1$ are taken. The initial values obey the uniform distribution, that is, $x_0 \sim \operatorname{Uni}(-2,2)$ and $y_0 \sim \operatorname{Uni}(-2,2)$, are randomly selected in the bounded domain. 

\begin{figure}[htbp]
    \centering
    \includegraphics[scale = 0.6]{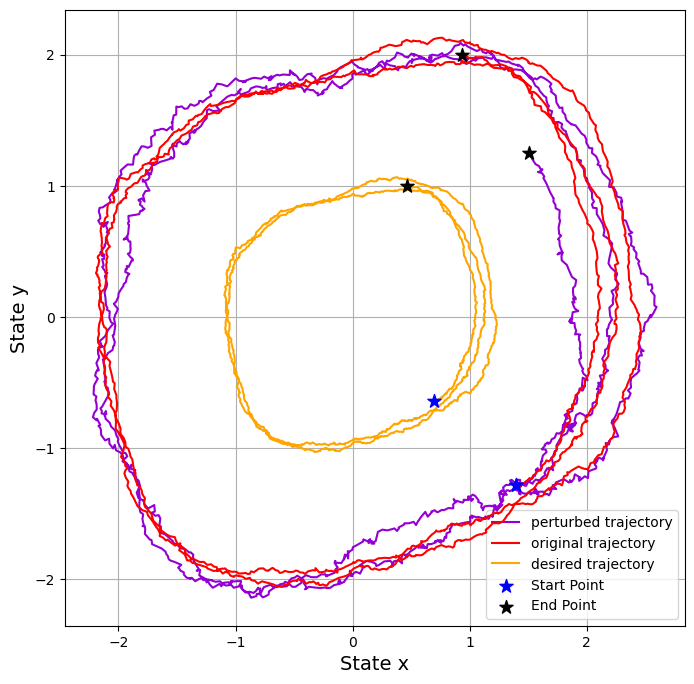}
    \caption{Phase portrait of the original, perturbed, and desired trajectories under low-intensity noise without multiple time scales. (Red) Original trajectory. (Yellow) A simplified desired trajectory used to illustrate the essence of the control problem (not the actual desired trajectory used in the experiment). (Purple) Perturbed trajectory, which serves as training data for the stochastic NG-RC controller to learn the influence of the external input $u$ on the stochastic dynamical system. Blue stars indicate the initial points of the three trajectories after discarding the first 500 sample points. Black stars indicate their respective termination points."}
    \label{fig:2d_add data}
\end{figure}

To illustrate our experimental goal, we first consider a simple setup that captures the essence of our control task (as shown in Fig.\ref{fig:2d_add data}). An Euler-Maruyama discretizations is performed on the system \eqref{VDPadd} to obtain a 2D trajectory of the Original data with 2000 sample points, denoted as the \textbf{original trajectory}. A \textbf{desired trajectory} is then designed as a modification of this original trajectory. Then we introduce a randomly generated perturbation into the discretized SDE, producing the \textbf{perturbed trajectory}. The objective of the controller is to track the desired trajectory despite stochastic perturbations and unknown trajectory transitions. 
In particular, the "Start Point" in Fig.\ref{fig:2d_add data} represents the 500th sample point because we are more concerned with the sudden change of the system in the steady state.

In the following, we split the control task into two steps based on the dynamic evolution and control law design of the stochastic NG-RC model.

\textbf{Step 1: learning data dynamics via stochastic NG-RC.} To train a stochastic NG-RC controller, two independent perturbation signals $u_x, u_y$ (the \textbf{top-right} panel of Fig.\ref{fig:2d_add train}) are randomly generated and injected into the numerical iteration of the controlled system in equation \eqref{cVDPadd}. The resulting perturbed trajectory is shown in the left panel of Fig.\ref{fig:2d_add train}. We design a desired trajectory incorporating two transition phases with distinct modulations of amplitude: outward expansion and inward contraction, as depicted in the left panel of Fig.\ref{fig:2d_add train}. This design evaluates the adaptability of the stochastic NG-RC controller under different amplitude modulations and assesses its effectiveness across varying dynamical regimes. In addition, it simulates complex control requirements found in real-world systems.

\begin{figure}[htbp]
    \centering
    \includegraphics[width=\textwidth]{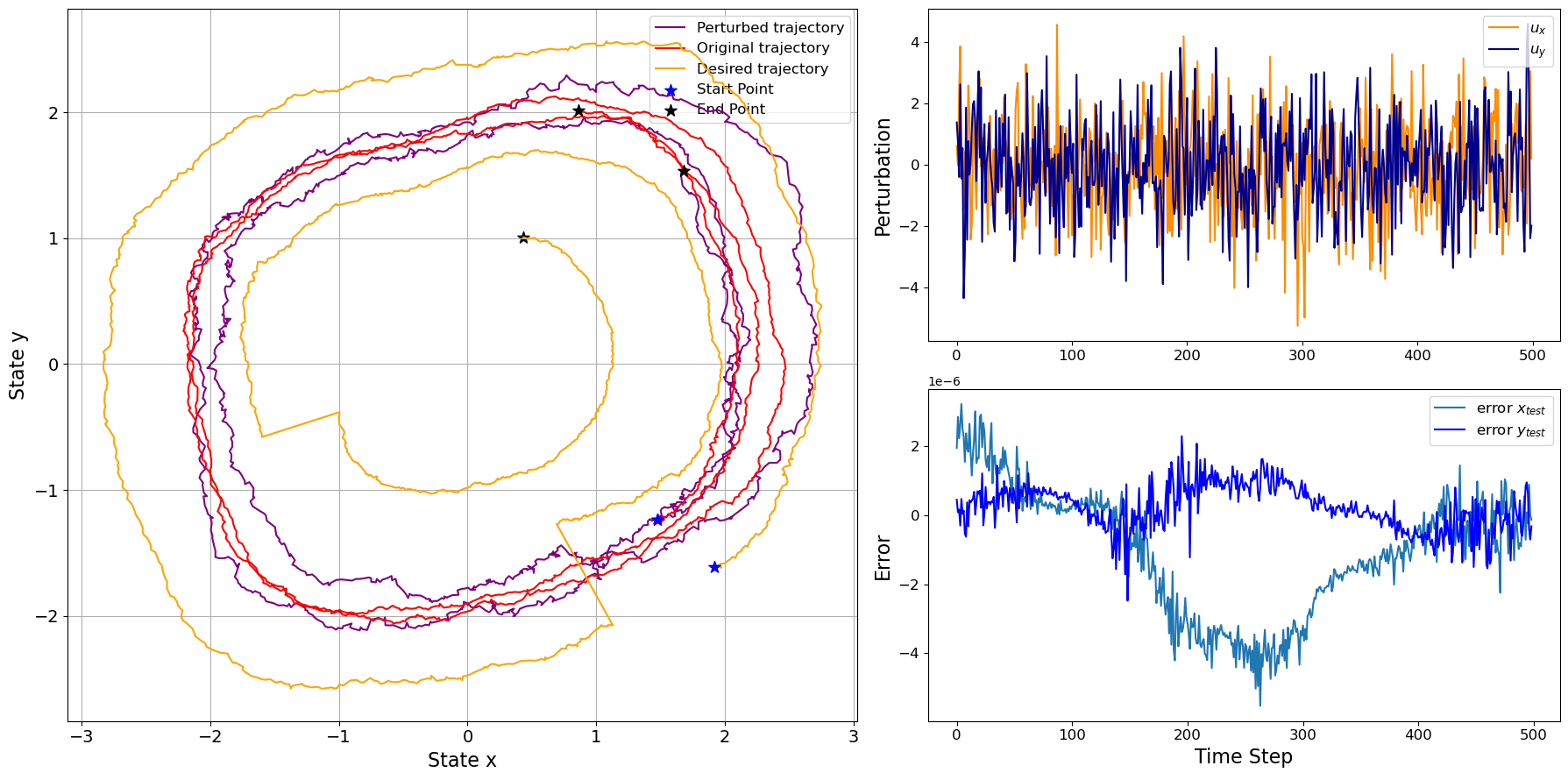}
    \caption{Training results for stochastic NG-RC in low-intensity noise scenarios without multiple time scales.
    The Fig.consists of three subplots. The \textbf{left} panel displays the phase portrait of the original, perturbed, and designed desired trajectories under low-intensity noise without multiple time scales. Blue stars indicate the starting points of the trajectories after discarding the first 500 sample points, while black stars mark their termination points. The \textbf{top-right} panel shows the randomly generated perturbation signals $u_x, u_y$ and the \textbf{bottom-right} panel presents the stepwise prediction error of the stochastic NG-RC model over 500 test samples.}
    \label{fig:2d_add train}
\end{figure}

The left panel of the Fig.\ref{fig:2d_add train} illustrates the interplay between the three trajectories. For clarity, we discard the initial 500 sample points and display only the subsequent 1500 stabilized samples. The (red) original trajectory exhibits a stable limit cycle under low-intensity noise without multiple time scales. The (purple) perturbed trajectory oscillates around the original system's limit cycle but does not inherently cover the steady states that the (yellow) desired trajectory follows. We partition the 1500 perturbed samples into 1000 training samples and 500 test samples. 

Through training, we obtain the optimal ridge regression parameter $\alpha = 7.7426\times 10^{-7}$, and The predicted root mean squared error(RMSE) for the test samples is $2.2412\times 10^{-6}$. The \textbf{bottom-right} panel of the Fig.\ref{fig:2d_add train} only shows how well the stochastic NG-RC model predicts the completely unknown test data and the prediction error on each time step (order of magnitude in $10^{-6}$).

\textbf{Step 2: designing control law for stochastic dynamical systems. } Based on the desired trajectory (illustrated in yellow in Fig. \ref{fig:2d_add train}), the system's external inputs (feedback control signals) $u^x$ and $u^y$ are derived by inverse-solving the stochastic NG-RC controller. The following analysis presents the control performance of the stochastic NG-RC controller under low-intensity noise, without the influence of multiple time scales.

\begin{figure}[htbp]
    \centering
    \includegraphics[width=\textwidth]{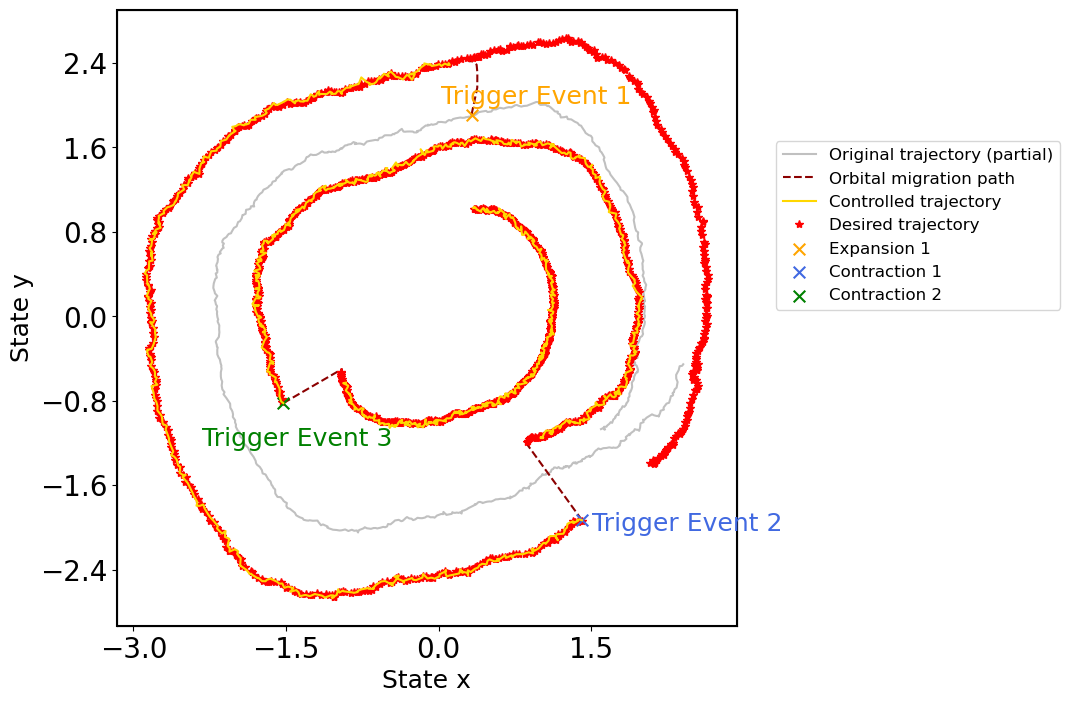}
    \caption{Phase-space tracking performance of the stochastic NG-RC controller in low-intensity noise scenarios without multiple time scales.\\
    a) Grey curve: Original Van der Pol (VDP) system trajectory without control.\\
    b) Yellow marker "X": Initial control activation point(1.2× the original amplitude). \\
    c) Dark red dashed line: Controlled trajectory converging toward desired trajectory (red).\\
    d) Blue and green markers "X": Trigger events for trajectory contraction (0.8× and 0.5× amplitude scaling, respectively).}
    \label{fig:2d_add control1}
\end{figure}

As depicted in Fig. \ref{fig:2d_add control1}, the controller initiates tracking from the yellow marker "X" (control start). Within 1-2 time steps, the controlled system rapidly approaches the desired trajectory (red) along the dark red dashed line, under the stochastic NG-RC controller. Detailed temporal evolution of control signals and error profiles are provided in Supplementary material Fig. \ref{fig:2d_add control}, including time-stamped trajectory transitions (black dashed lines), variable-specific tracking results ( $x$ : blue vs. $y$ : gold curves), and control error dynamics $\left(e_i \in \mathbb{R}^2\right)$ across the entire process.

The desired trajectory reflects the principles of the Self-ETC, wherein the controller operates without prior knowledge of exact timing for trajectory transitions. As shown in Supplementary material Fig. \ref{fig:2d_add control}, the controller first tracks an outward expansion (1.2× the original amplitude). Subsequently, it undergoes two unknown transitions—first to 0.8 times and then to 0.5 times of the original amplitude. These two contractions serve as \textbf{Trigger Event 2} and \textbf{Trigger Event 3}, respectively.
At these critical moments, the controller detects trajectory deviations by continuously calculating the control error $e_i=x_i-x_{\mathrm{des}, i}$. Once the error exceeds a threshold, the controller adjusts the input to ensure adaptive tracking. This error-driven mechanism embodies event-triggered control, offering both flexibility and efficiency while maintaining robustness amid uncertain trajectory changes.

The third row of Fig. \ref{fig:2d_add control} reveals that when a large discrepancy exists between the desired trajectory and the current system observation (akin to a full-observation system state), the controller outputs a significant control signal. We observe the control signals at the moments corresponding to the three black dashed lines: $u_{t_1} = (12.82,25.77),u_{t_2} = (-45.54,77.73), \text{and\ } u_{t_3} = (59.33,24.76)$. Later control signal magnitudes for these instances fall within the interval$ [-3, 15]$. In addition, the fourth row of Fig. \ref{fig:2d_add control} presents the control error at each moment, with an overall $RMSE = 0.1654$ for the entire control period. These results demonstrate that the stochastic NG-RC controller can rapidly adjust to changes in the control objective and accurately steer the system towards the desired trajectory within the first one to two steps, under low-intensity noise without multiple time scales.

\paragraph{Comparative Experiments with Conventional NG-RC Controllers}

As outlined in section \ref{sec:model}, a key design of the stochastic NG-RC controller is the incorporation of noise feature vector. The following presents a comparison of the conventional NG-RC with the stochastic NG-RC in system \eqref{cVDPadd}, using noise intensities $\sigma_1 = 1,\sigma_2 = 2$ and time scale $\epsilon = 0.5$, respectively.
\begin{figure}[htbp]
    \centering
    \includegraphics[width=\textwidth]{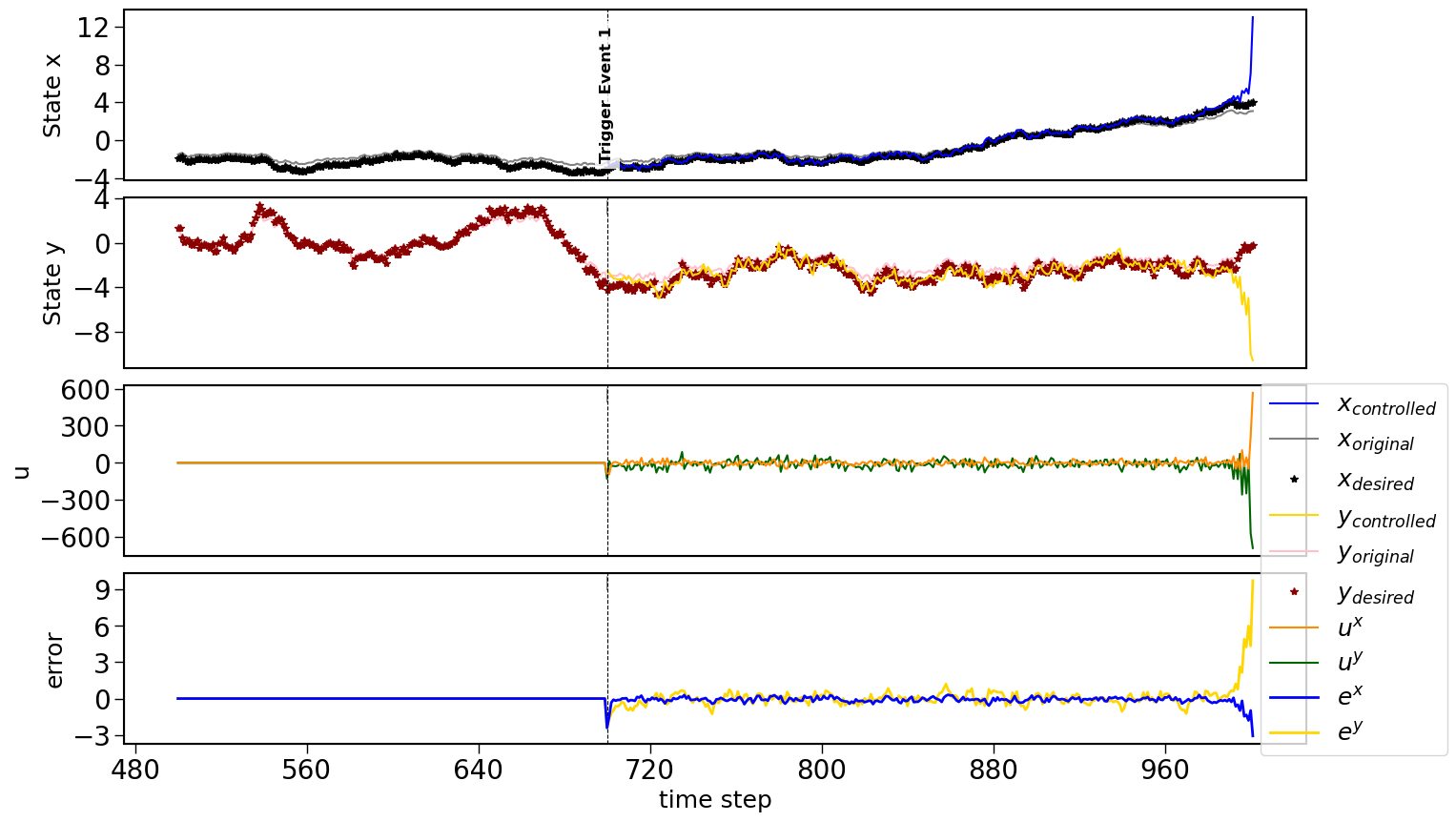}
    \caption{Control results of the conventional NG-RC at $\sigma_1 = 1,\sigma_2 = 2,\epsilon = 0.5$. The controller diverges at the 302th step after the start of control.}
    \label{fig:baseline}
\end{figure}
\begin{figure}[htbp]
    \centering
    \includegraphics[width=\textwidth]{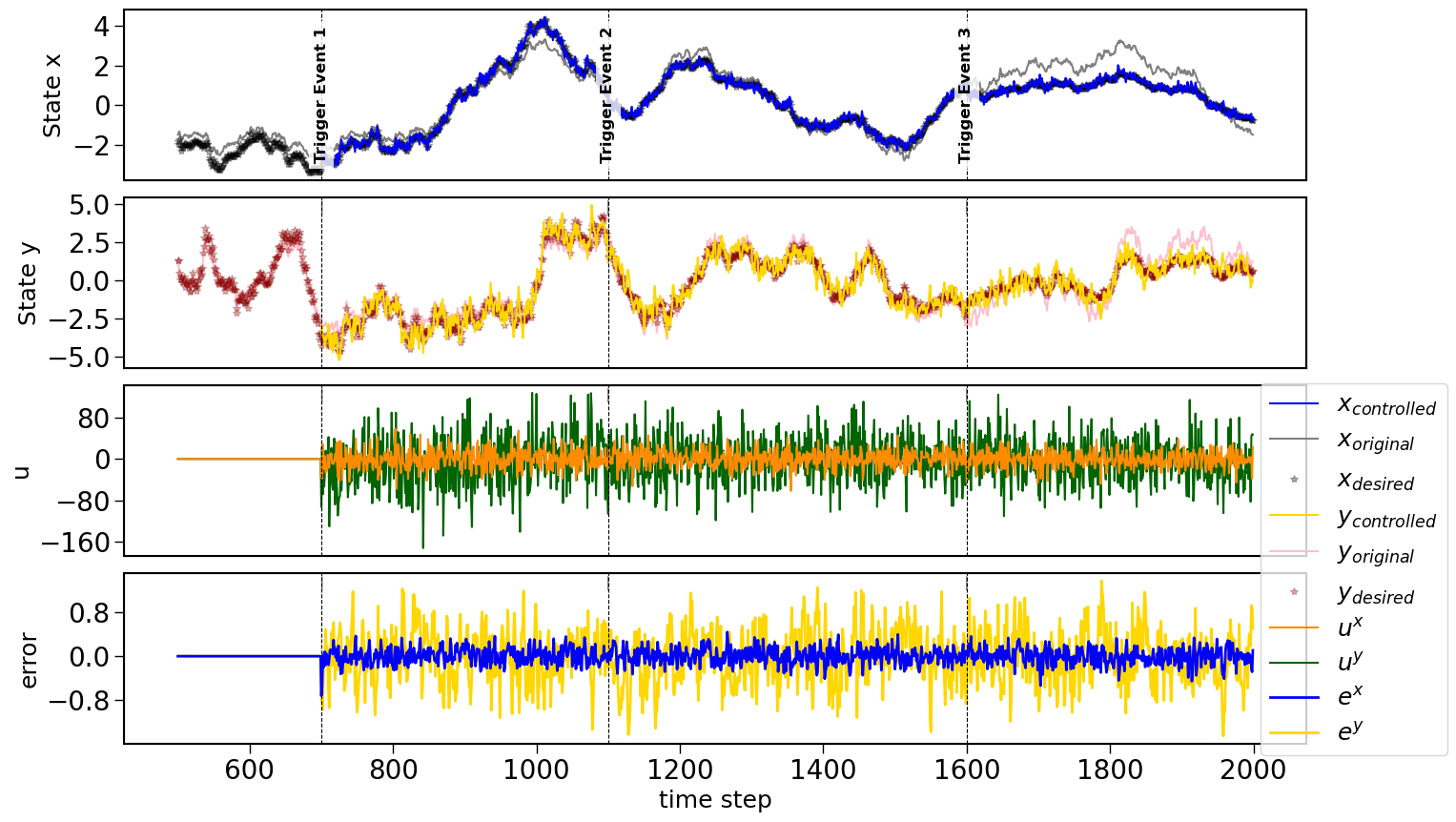}
    \caption{Control results of the stochastic NG-RC at $\sigma_1 = 1,\sigma_2 = 2,\epsilon = 0.5$. The control signal remains within the interval [-160,160].}
    \label{fig:add_1_2}
\end{figure}

As shown in Fig. \ref{fig:baseline} and Fig. \ref{fig:add_1_2}, when tested with system parameters $\sigma_1 = 1,\sigma_2 = 2,\epsilon = 0.5$, the conventional NG-RC controller diverges at the 302th time step, while the signal outputed from the stochastic NG-RC controller maintains within the range of [-160,160]. The overall control performance, evaluated by the total $RMSE = 0.3632$, shows that the control error remains within [-1.6, 1.6] for the entire interval, as seen in the fourth row of Fig. \ref{fig:add_1_2}. The desired trajectory is the same as in the previous experiment, featuring a deflation of the original trajectory. Notably, under the influence of larger noise intensities and multiple time scales, the primary cause of the discrepancy between the desired and controlled trajectories is no longer due to the design of the desired trajectory itself. This is evident in the control signals $u^x, u^y$, as seen by comparing the third row of Figure \ref{fig:2d_add control} with Figure \ref{fig:add_1_2}.

It is worth noting that the conventional NG-RC used for comparison employs a one-step delay embedding in its feature vector $O_{\text {total }}$, which includes the linear term $O_{\operatorname{lin}, i}= \left\{x_{i-1}, y_{i-1}, x_i, y_i\right\}$. A more sophisticated NG-RC design incorporating higher-order delay embeddings could potentially enhance control performance; however, this direction remains beyond the scope of the present work.

\paragraph{Robustness Analysis of the Controller to Noise Intensity and Time Scales}

In the previous subsection, we demonstrated through experimental results that the conventional NG-RC controller loses its robustness when subjected to high noise intensity and multiple time scales. This raises the important question: how do noise intensity and time scales affect the control performance of the stochastic NG-RC model in a controlled SDE system (\ref{cVDPadd}) with additive noise? To explore this, we conducted experiments with twelve different system parameter settings, repeating each experiment more than five times. The resulting RMSE values were averaged and presented in Fig. \ref{fig:noise_robustness} and Fig. \ref{fig:noise_robustness1}.

\begin{figure}[htbp]
    \centering
    \includegraphics[width=\textwidth]{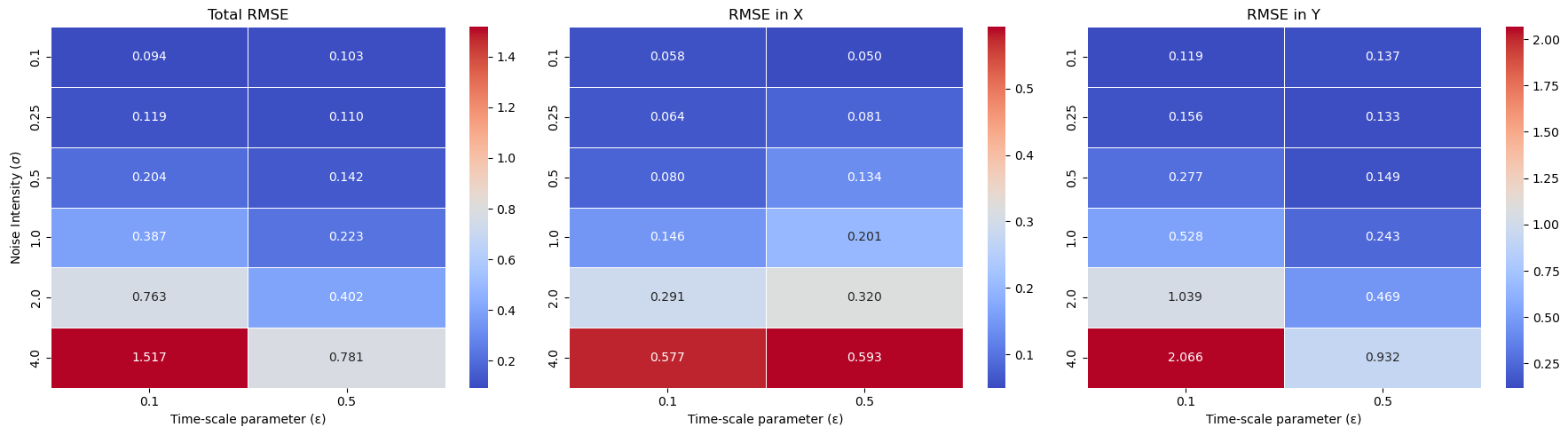}
    \caption{Noise robustness analysis of the stochastic Van der Pol system under varying noise intensity $(\sigma)$ and time-scale parameters $(\varepsilon)$.
    Left: Total RMSE heatmap. Middle \& Right: RMSE in X and Y components, respectively.
    }
    \label{fig:noise_robustness}
\end{figure}
\begin{figure}[htbp]
    \centering
    \includegraphics[width=\textwidth]{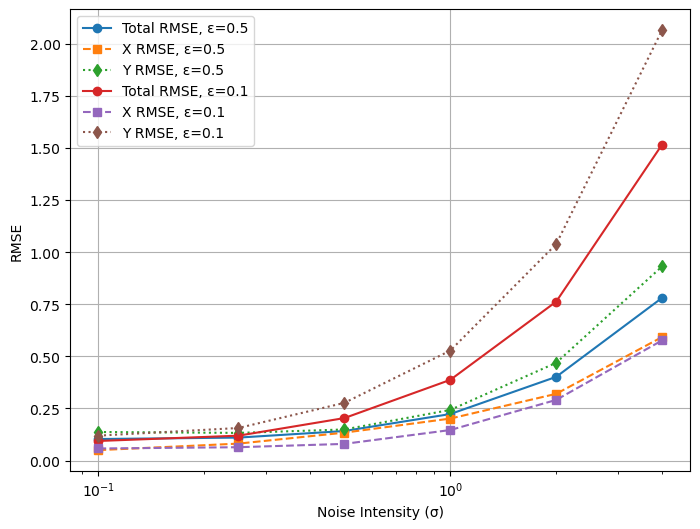}
    \caption{
    Control error sensitivity of the stochastic Van der Pol system under varying noise intensity ( $\sigma$ ) and time-scale separation ( $\varepsilon$ ). Total RMSE (solid lines), X RMSE (dashed lines), and Y RMSE (dotted lines) are plotted for $\varepsilon=0.5$ and $\varepsilon=0.1$ .
    }
    \label{fig:noise_robustness1}
\end{figure}

The robustness of the stochastic NG-RC controller is fundamentally governed by the time-scale parameter $\varepsilon$ and noise intensity $\sigma$. As shown in Fig. \ref{fig:noise_robustness1}, when $\varepsilon=0.1$ (strong time-scale separation), the Y RMSE (dotted lines) surpasses 1.0 at $\sigma=2$, signaling a significant loss in control performance. This is consistent with the high-error region seen in the heatmap (Fig. \ref{fig:noise_robustness}, right panel), which highlights that the fast variable ($y$) is more susceptible to noise-induced deviations. In contrast, the $X$ RMSE (dashed lines) remains stable, reflecting the slower dynamics of the $X$ component, which exhibits linear drift.

\subsubsection{Case 2: Multi-scale SDEs under multiplicative noise}

In the construction of the stochastic NG-RC controller presented in section \ref{sec:model}, we simplified the SDE to a system where the diffusion term is a constant matrix \eqref{sde2}. Instead of explicitly controlling the diffusion term, the noise feature vector are employed to realise this control implicitly. In the stability analysis, we demonstrate that Lassalle's theorem holds under certain conditions for the functions \(f\), \(g\), \(\phi_1\), and \(\phi_2\). The results below show that the stochastic NG-RC controller is also robust to stochastic Van der Pol systems with multiplicative noise.

Consider a simple multiplicative noise model, where system \eqref{cVDPadd} can be rewritten as:
\begin{equation}\label{VDPmul}
\left\{\begin{array}{l}
d x = [-y + u_1] dt + \sigma_1 x dW_t^1, \\
d y = \left[\frac{1}{\varepsilon} \left(y - \frac{1}{3} y^3 + x\right) + u_2 \right] dt + \frac{\sigma_2}{\sqrt{\varepsilon}} y dW_t^2.
\end{array}\right.
\end{equation}

\begin{figure}[htbp]
    \centering
    \includegraphics[width=\textwidth]{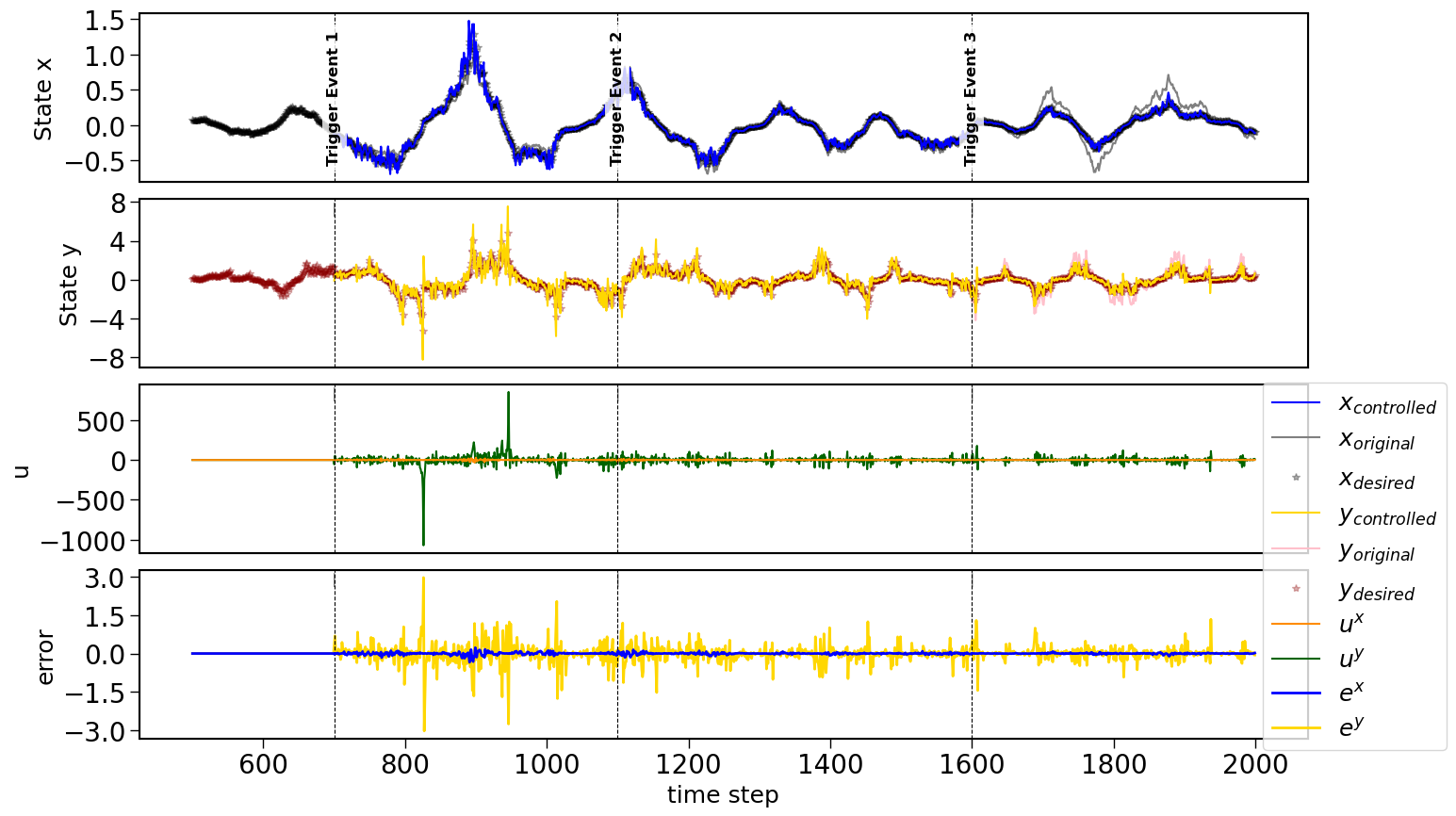}
    \caption{Control results for stochastic NG-RC under multiplicative noise.}
    \label{fig:mul}
\end{figure}

For the system \eqref{VDPmul}, with parameters \(\sigma_1 = 0.8\), \(\sigma_2 = 1\), and \(\varepsilon = 0.1\), the tracking results of the stochastic NG-RC controller are shown in Fig. \ref{fig:mul}. Comparing the first and second row , we observe that the controlled trajectory for the slow variable ($x$) (blue curve) tracks the desired trajectory (black pentagram) closely, while the controlled trajectory for the fast variable ($y$) (gold curve) exhibits significant oscillations near its desired trajectory (red pentagram). 

In the third row, the controller applies a stronger control signal (green curve) to the fast variable ($y$), which is required under the combined influence of multiplicative noise and multiple time scales. This ensures effective control despite the noise. Comparing the control performance in the four regions delineated by the three black dashed lines, we find that the control is relatively poorer in the second segment ($[t_1,t_2] = [700, 1100] $), where the desired trajectory is 1.2 times the amplitude of the original trajectory. This contrasts with the subsequent contraction scenarios (0.8 and 0.5 times, respectively), where the control performance improves (see fourth row: control error at each moment). 

This result is natural because, when the trajectory is expanding outward, the system must overcome its inherent attraction to the original trajectory, requiring additional positive control forces. In this scenario, the controller must counteract both the system's intrinsic dynamics and the noise to keep the trajectory within the target range. However, since the control outputs of the stochastic NG-RC are learned from historical data, tracking an over-expanded desired trajectory can lead to large control input amplitudes, which exacerbate the system's nonlinear effects and make control more challenging. Additionally, the expanded desired trajectory may enter regions of the state space that are less explored by the system, limiting the controller's generalization ability and thus reducing control accuracy.

In contrast, when the trajectory contracts, the desired trajectory stays within the system's domain of attraction, allowing the controller to apply a moderate negative control signal that reinforces the system's natural convergence tendencies. Consequently, the control signal amplitude is smaller, the system's nonlinear effects are less pronounced, and the control remains relatively stable. 

Finally, we record the total RMSE over the entire control interval, which is \(RMSE = 0.2359\).

\subsection{Epilepsy Adaptive Control Based on stochastic NG-RC}

In this section, we apply the stochastic NG-RC controller to real Electroencephalogram (EEG) data. Compared to the simulation experiments on the stochastic Van der Pol system, the study presented here is more challenging in the following aspects:

\begin{enumerate}
    \item \textbf{Complexity of real EEG signals:} In the simulation study, we utilized explicit stochastic differential equations (SDEs) to describe the system evolution. However, for EEG data, no explicit dynamical system equations are available. Consequently, we need to provide an approximate model for the EEG data from the perspective of the neurodynamic system inverse problem.
    
    \item \textbf{dynamical changes in pathological states:} There are significant differences in the neurodynamic mechanisms before and after seizures, manifested through changes in the system's structure and alterations in the effective noise characteristics. Conventional control methods typically assume stable or predictable system dynamics, but the time-varying nature of EEG data necessitates a more adaptive controller. In particular, noise is not merely a small perturbation but a key factor that influences the evolution of the system at different stages of the seizure.
\end{enumerate}

\subsubsection{Governing Law Learning of Epileptic EEG Data}
\label{sec:identify}

We analyzed 23-channel electroencephalography (EEG) recordings from a patient with focal epilepsy, as shown in the \textbf{upper panel} of Fig.~\ref{fig:1dEEG control}. EEG signals capture the electrical activity of the brain during seizures via electrodes placed on the scalp. The dataset comprises both pre-ictal and ictal phases, recorded as a time series matrix with a sampling rate of 256 Hz. Each channel was normalized to the interval $[-0.5,0.5]$ and subsequently downsampled by averaging every 16 points, yielding a dataset with a temporal resolution of $\delta t = 0.0625$. As illustrated in Fig.~\ref{fig:1dEEG control}, a neurologist-annotated transition point at $t=600$ separates pre-seizure and seizure states.

We model the evolution of EEG signals using a one-dimensional stochastic differential equation (SDE)
\begin{equation}\label{eq:sde}
d X_t = f\left(X_t\right) d t + g\left(X_t\right) d W_t, \quad X_0 = x_0,
\end{equation}
where $f(X_t)$ and $g(X_t)$ represent the drift and diffusion terms, respectively, and $W_t$ denotes a standard Brownian motion. We assume that $f(\cdot)$ and $g(\cdot)$ are twice continuously differentiable with uniformly bounded derivatives in $\mathbb{R}$. The Kramers-Moyal expansion provides a direct means of estimating these terms from short-term data, yielding a data-driven SDE representation
\begin{equation}
d x_t=f_x\left(x_t, \theta\right) d t+g_x\left(x_t, \theta\right) d W_t.
\end{equation}
Here, $x_t$ represents the system state, and $\theta$ is the set of coefficients associated with an orthogonal basis expansion (see Table~\ref{tab:1dEEG}). Following a sparse identification of nonlinear dynamics (SINDy) approach, we employ a polynomial basis $\boldsymbol{\Theta}(x) = \{1, x, x^2, x^3\}$ to approximate the drift and diffusion terms:
\begin{equation}\label{eq:km_expansion}
\begin{gathered}
f_x\left(x_t, \theta\right) = \lim _{\Delta t \rightarrow 0} \mathbb{E}\left[\left.\frac{\left(X_{t+\Delta t}-X_t\right)}{\Delta t} \right\rvert\, X_t=x_t\right] = \theta_1 \boldsymbol{\Theta}(x),\\
g_x^2\left(x_t, \theta\right) = \lim _{\Delta t \rightarrow 0} \mathbb{E}\left[\left.\frac{\left(X_{t+\Delta t}-X_t\right)^2}{\Delta t} \right\rvert\, X_t=x_t\right] = \theta_2 \boldsymbol{\Theta}(x).
\end{gathered}
\end{equation}
The coefficients $\theta_1$ and $\theta_2$ are estimated using sparse regression via Lasso optimization. Specifically, given a time series $X = [X_{t_1}, \ldots, X_{t_m}] \in \mathbb{R}^{1 \times m}$, where $m=1500$ represents the temporal window length (Fig.~\ref{fig:1dEEG control}), we minimize the following loss functions:
\begin{equation}
\begin{gathered}
\mathcal{L}_{\text{drift}}(\theta_1) = \frac{1}{(m-1)} \sum_{i=1}^{m-1} \left\| \frac{X_{t_{i+1}}-X_{t_i}}{t_{i+1}-t_i} - f_x\left(x_t, \theta_1\right) \right\|^2 + \alpha \|\theta_1\|_1, \\
\mathcal{L}_{\text{diffusion}}(\theta_2) = \frac{1}{(m-1)} \sum_{i=1}^{m-1} \left\| \frac{(X_{t_{i+1}}-X_{t_i})^2}{t_{i+1}-t_i} - g_x^2\left(x_t, \theta_2\right) \right\|^2 + \beta \|\theta_2\|_1.
\end{gathered}
\end{equation}
Here, $\alpha$ and $\beta$ are sparsity-inducing thresholds: coefficients satisfying $|\theta_j^i| < \alpha$ or $\beta$ during training are set to zero($ i \in\{1,2,3,4\}, j \in\{1,2\}$), ensuring a parsimonious representation. The final identified coefficients are summarized in Table~\ref{tab:1dEEG}.

\begin{table}[htbp]
\centering
\caption{Sparse coefficient estimates for the drift and diffusion terms}
\label{tab:1dEEG}
\begin{tabular}{|c|c|c|}
\hline
\textbf{Basis Function} & \textbf{Drift Term $\theta_1$} & \textbf{Diffusion Term $\theta_2$} \\ \hline
$1$         & 0 & 0       \\ \hline
$x_t$     & -5.8878 & -0.1150     \\ \hline
$x^2$     & 0 & 4.0277       \\ \hline
$x^3$     & 0 & 0.0375      \\ \hline
\end{tabular}
\end{table}

\begin{figure}[htbp]
    \centering
    \includegraphics[width=\textwidth]{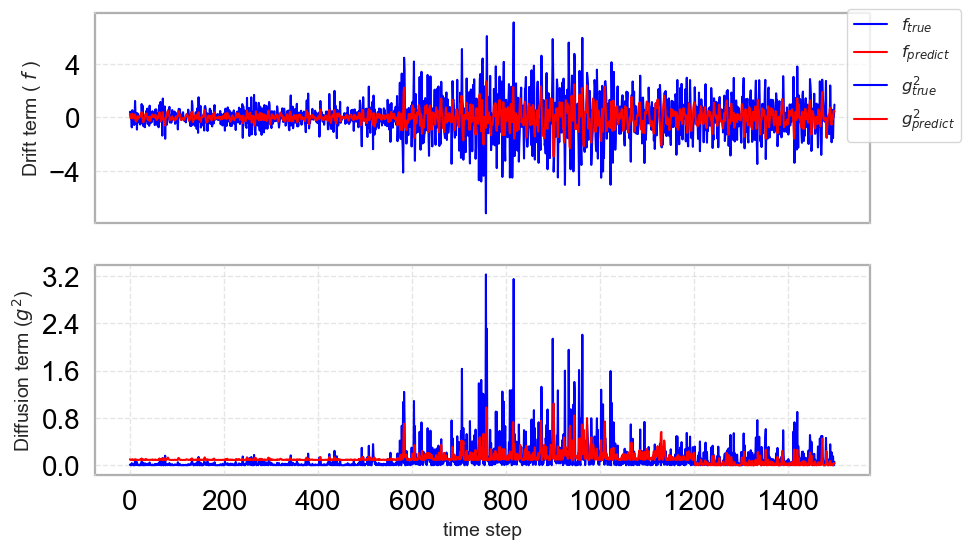}
    \caption{Prediction performance of the Kramers-Moyal network on the governing law of EEG signals.  
    (Top) Prediction of the drift term $f$ by the Kramers-Moyal network: the blue curve ($f_{\text{true}}$) represents the ground truth, while the red curve ($f_{\text{predict}}$) represents the predicted values.  
    (Bottom) Prediction of the squared diffusion term $g^2$: the blue curve ($g_{\text{true}}^2$) represents the ground truth, while the red curve ($g_{\text{predict}}^2$) represents the predicted values.}
    \label{fig:sindy_train}
\end{figure} 

\begin{figure}[htbp]
    \centering
    \includegraphics[width=\textwidth]{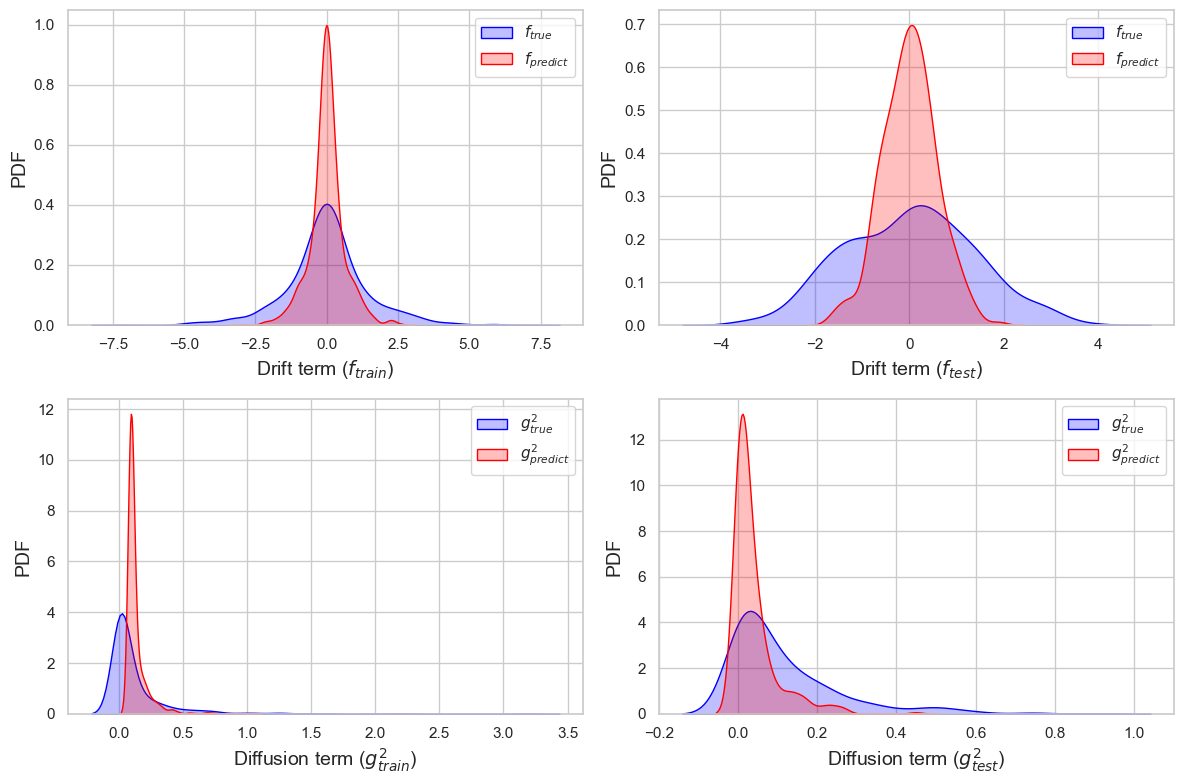}
    \caption{Kernel Density Estimation (KDE) comparison between true and predicted data.  
    The KDE plots show the probability density functions (PDFs) of the true data (blue) and the predicted data (red) for both drift and squared diffusion terms, providing insight into the distributional deviations of the model predictions.}
    \label{fig:KDE}
\end{figure} 

For experimental validation, 1500 drift and diffusion estimates ($f_{\text{true}}, g^2_{\text{true}}$) were computed using Eq.~\eqref{eq:km_expansion}. Fig.~\ref{fig:sindy_train} illustrates the performance of the Kramers-Moyal network in predicting the drift ($f_x$) and squared diffusion ($g^2_x$) terms of the EEG signal's governing law. The first 1200 time steps correspond to training data, while the last 300 steps are test data. The optimal sparsity threshold $\alpha,\beta$ was determined via cross-validation using LassoCV from the \texttt{sklearn} package, yielding $\alpha=0.0012$ and $\beta = 10^{-5}$.

From Fig.~\ref{fig:sindy_train}, we observe that in the low-amplitude regime (first 600 time steps), the model accurately tracks the true dynamics, with the predicted curves nearly overlapping the ground truth. However, in the high-amplitude regime (from time step 600 onward), particularly between 750 and 1000, the squared diffusion term $g_{\text{true}}^2$ exhibits large fluctuations, while the predicted values $g_{\text{predict}}^2$ remain relatively suppressed, failing to capture the extreme variations associated with seizure dynamics. Given the low-dimensional basis expansion, sparsity constraints were not stringent but are expected to play a more significant role in higher-dimensional SDE learning. To further analyze the prediction errors, Fig.~\ref{fig:KDE} presents Kernel Density Estimation (KDE) plots of the probability density functions (PDFs) for the true and predicted data.

By analyzing Figs.~\ref{fig:sindy_train}--\ref{fig:KDE}, we can assess the limitations of the Kramers-Moyal operator network from a dynamical systems perspective. The primary limitations arise from dual sources of model bias:  
(1) The existing model reduces multi-regional coupling to an isolated system, disregarding the topological dynamics of the connectome. The 23-channel EEG data is essentially a stochastic projection of a complex corticothalamic network, and the absence of cross-spectral coupling terms prevents the model from capturing synchrony patterns seen in Neural Mass Models (NMM).  
(2) The stochastic modeling paradigm may be inadequate. Traditional Brownian-motion-driven stochastic differential equations struggle to capture the spike-wave propagation properties of epileptic discharges, conflicting with the statistical properties of extreme events described by $\alpha$-stable Lévy processes. The latter may offer advantages in modeling critical state transitions in cortical reticular systems, where heavy-tailed distributions align with high-frequency oscillations characteristic of seizure events.

Despite these challenges, the proposed method provides a foundation for studying stochastic dynamics in EEG signals. While many existing models filter out noise as observational error, our approach treats noise as an intrinsic component of the system dynamics. The identified governing law serves as the foundation for subsequent control design, providing a generative model for perturbed trajectory and governing the system's response under control, as detailed in the \textbf{bottom panel} of Fig.~\ref{fig:1dEEG control}.

\subsubsection{Seizure Suppression Simulation Using stochastic NG-RC}

To validate the effectiveness of the stochastic NG-RC in real EEG control tasks, we conduct an EEG control experiment following the procedure outlined in Section \ref{sec:experiment}, using the results in Table \ref{tab:1dEEG}. Although EEG signals are significantly more complex than those in the simulated environment, our experimental results demonstrate successful control, suggesting that the stochastic NG-RC possesses generalization capabilities for complex real-world signals. This provides a feasible data-driven approach for EEG modulation.

\begin{figure}[htbp]
    \centering
    \includegraphics[width=\textwidth]{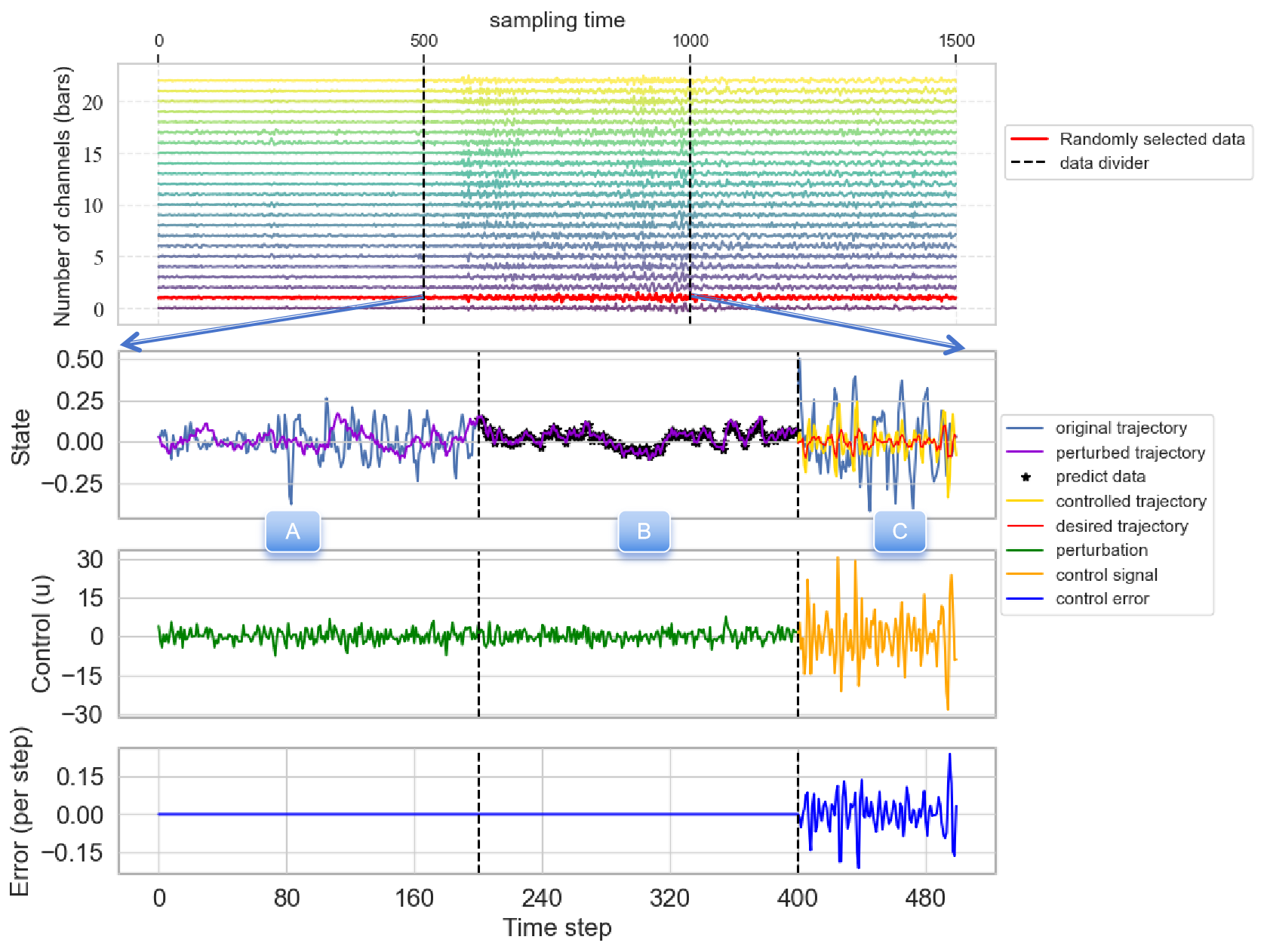}
    \caption{Control of one-dimensional EEG data using the stochastic NG-RC controller. 
    \textbf{Upper panel:} 23-channel EEG data, with one randomly selected time series (marked in red) used for learning a governing law. The sequence is divided into three equal segments by black dashed lines, where the middle 500 samples correspond to seizure data targeted for control. 
    \textbf{Bottom panel:} Three-stage control process for the selected EEG time series. 
    \textbf{Phase A:} Constructing the SDE with SINDy and generating a perturbed trajectory. 
    \textbf{Phase B:} Training the stochastic NG-RC controller with the perturbed data. 
    \textbf{Phase C:} Applying the trained controller to drive seizure-state data toward the resting-state target. 
    The three rows represent the system state $X$, control signal $u$, and control error $e$.}
    \label{fig:1dEEG control}
\end{figure}

\paragraph{Data Generation and Training}



The dataset consists of 1500 EEG samples, where the 501st-1000th points serve as the original trajectory (Figure \ref{fig:1dEEG control}, bottom panel, blue), containing 100 resting-state samples and 400 seizure-state samples. The first 500 points, all in the resting state, define the desired trajectory. Without additional smoothing, we formulate the one-dimensional controlled stochastic differential equation (SDE) as follows
\begin{equation}\label{EEG}
d X_t = [f_{\theta_1}\left(X_t\right)+u_t] d t + g_{\theta_2}\left(X_t\right) d W_t, \quad X_0 = x_0.
\end{equation}
Here, the drift and diffusion terms, $f_{\theta_1}$ and $g_{\theta_2}$, are represented by the basis functions in Table \ref{tab:1dEEG} with their corresponding coefficients. To train the stochastic NG-RC controller, a random perturbation signal $u_t$ is introduced into the numerical iteration of \eqref{EEG}, generating the perturbed trajectory (purple curve in Figure \ref{fig:1dEEG control}).

To train the stochastic NG-RC controller, a random perturbation signal $u_t$ is introduced into the numerical iteration of Eq.~\eqref{EEG}, generating the perturbed trajectory (purple curve in Fig.~\ref{fig:1dEEG control}). The bottom panel (left two segments: Phase A and B) illustrates this process. This phase highlights how the perturbation $u_t$ modifies the system's dynamics, shifting it away from the high-amplitude, irregular oscillations characteristic of seizure states.

In phase A, the perturbed trajectory is then used to train the stochastic NG-RC model. A ridge regression parameter of $\alpha = 0.1931$ is optimized. The next 200 time steps (Phase B) validate the trained NG-RC controller by evaluating its one-step prediction accuracy on unseen perturbed data. The black crosses in Fig.~\ref{fig:1dEEG control} indicate predicted values, which closely align with the perturbed trajectory (purple), with a prediction RMSE of $0.1331$. These results confirm that the stochastic NG-RC model effectively captures the short-term evolution of the controlled SDE and generalizes well to previously unseen perturbations.

\paragraph{Control Results}

In the control stage, the trained stochastic NG-RC controller generates feedback control signals $u_t$ to drive the seizure-state data toward the desired resting-state trajectory. The rightmost segment of the bottom panel (Phase C) in Fig.~\ref{fig:1dEEG control} presents the control outcomes. The red curve represents the desired resting-state trajectory, while the yellow curve shows the controlled trajectory. The blue curve, representing the original seizure trajectory, serves as a reference for comparison.

During the initial phase of control ($t=400$ to $t=410$), the controlled trajectory remains close to the original seizure data. However, around $t=410$, the controlled trajectory begins deviating from the seizure state and moves toward the desired trajectory. Between $t=440$ and $t=480$, the controlled trajectory successfully suppresses large-amplitude oscillations, demonstrating effective stabilization. Nevertheless, near $t=490$, the control performance deteriorates, with the trajectory exhibiting increased oscillations and divergence from the target state. The overall control performance is quantified by a long-time control RMSE of $0.0752$ over the 100-step control interval.

The bottom row of Fig.~\ref{fig:1dEEG control} presents the control error, highlighting the discrepancy between the controlled and desired trajectories. The increase in error towards the end of the control interval suggests potential limitations in the learned controller, warranting further refinement of the control strategy to enhance long-term stability. Several factors may contribute to the observed control errors:

\begin{enumerate}
    \item \textbf{Governing law learning error}: Given the complexity of EEG signals, the system \eqref{EEG} serves only as an approximation of the underlying true dynamics, introducing inherent model uncertainty during the governing law learning phase. Consequently, the controller may fail to predict the exact dynamic response when executing the control task. As discussed in section \ref{sec:identify}, potential improvements include incorporating more accurate neuro-dynamic models and employing non-Gaussian noise modeling techniques to enhance system representation.

    \item \textbf{Impact of perturbation signals}: The controller's learning process relies on the generated perturbed trajectory (e.g., the purple traces in Fig.~\ref{fig:1dEEG control}). However, the perturbation signal $u_t$ is randomly generated, affecting both the dynamics embedded in the perturbed data and the learning outcome of the controller. Experimental results indicate that different perturbations $u_t$ may lead the controller to learn distinct control strategies, significantly influencing the final control performance. This underscores the necessity of designing a more structured perturbation signal $u_t$ to optimize control effectiveness.

    \item \textbf{Effect of stochastic noise}: In simulating the discretized system \eqref{EEG}, the Euler-Maruyama method is employed, incorporating a stochastic noise term $\sqrt{\Delta t} \xi_t$, where $\xi_t \sim \mathcal{N}(0, I_N)$ represents an independent and identically distributed standard Gaussian white noise. Due to this random perturbation, the probability of the perturbed data aligning precisely with the original trajectory is exceedingly low. Consequently, the controller may fail to fully capture the dynamical structure of the original system during training. This implies that both the desired trajectory $x_{\text{desired}}$ and the original trajectory $x_{\text{original}}$ remain unknown to the stochastic NG-RC controller. Addressing the uncertainty induced by stochastic noise presents a critical direction for future optimization of control performance.

    \item \textbf{Accumulation of long-term control errors}: In the control phase (Phase C in the bottom panel of Fig.~\ref{fig:1dEEG control}), starting from $t=400$, the evolution of the controlled trajectory depends solely on its past states and the control input $u_t$ generated by the stochastic NG-RC. As a result, errors accumulate over time, potentially leading to error amplification, deviation from the desired trajectory, or even unstable oscillations in long-term iterations.
\end{enumerate}

\begin{figure}[htbp]
    \centering
    \includegraphics[scale = 0.5]{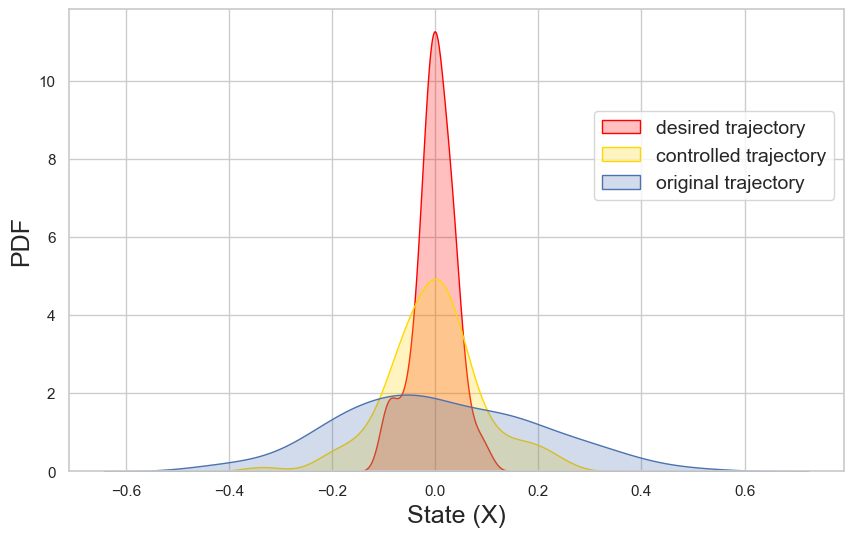}
    \caption{Kernel density estimation (KDE) comparison for long-term control under the stochastic NG-RC framework. The red, yellow, and blue regions correspond to the probability density functions of the desired trajectory, the controlled trajectory, and the original trajectory, respectively.}
    \label{fig:1dEEG control1}
\end{figure}

To further analyze control errors, we not only plotted the control error over time (last row of Fig.~\ref{fig:1dEEG control}), but also assessed the control performance from a probabilistic perspective. Fig.~\ref{fig:1dEEG control1} presents the kernel density estimation (KDE) of different trajectory distributions. The seizure-state distribution (blue) exhibits strong oscillatory characteristics, while the resting-state trajectory (red) shows a smaller amplitude, primarily concentrated in the range $[-1.5,1.5]$. Under the stochastic NG-RC controller, the controlled trajectory (yellow) exhibits reduced amplitude compared to the seizure state and trends towards the resting trajectory.

While preliminary experiments demonstrate that the proposed control strategy can mitigate seizure dynamics to some extent, further optimization is necessary for clinical applications. In particular, the amplitude of the control input must be strictly regulated to ensure safety in neuro-modulation interventions. This necessitates in-depth research in neuroscience, including strategies for designing control signals within physiologically safe limits to minimize potential risks to patients. Future studies should integrate interdisciplinary knowledge from neuroscience, biomedical engineering, and control theory to develop more robust and clinically viable control strategies.

\section{Main Findings and Conclusions}
\label{sec:conclusion}

In this study, we propose a next-generation reservoir computing (NG-RC)-based controller, referred to as stochastic NG-RC. This controller enables precise closed-loop control of complex stochastic dynamical systems. Furthermore, we analyze the controllability conditions of the controlled stochastic differential equations (SDEs) and establish control error convergence criteria using the stochastic LaSalle theorem, providing a theoretical foundation for closed-loop control of stochastic dynamical systems.

To evaluate the effectiveness and robustness of the stochastic NG-RC controller, we conduct both simulation experiments and real-data experiments.

\subsection{Robust Control of Stochastic Systems with Multiple Time Scales}
\begin{itemize}
    \item In the case of low-intensity noise without multiple time scales, we validate the rapid and precise control capability of stochastic NG-RC on a stochastic Van der Pol oscillator. The controlled system stabilizes within the desired trajectory in just 1--2 time steps.
    \item Compared to the method proposed by Kent et al.~(2024)\cite{kent2024controlling}, conventional NG-RC exhibits control failure under noise intensities $\sigma_1 = 1, \sigma_2 = 2$ and time scales $\epsilon = 0.5$. In contrast, stochastic NG-RC maintains control stability across the entire control interval, achieving a significantly lower joint RMSE of 0.3632.
    \item We further explore the robustness of stochastic NG-RC under various noise types (additive, multiplicative), noise intensities, and different time scales, demonstrating its broad applicability.
\end{itemize}

\subsection{Clinical Validation of Seizure Suppression}
\begin{itemize}
    \item The underlying stochastic dynamical system of epileptic seizures is reconstructed from EEG data using Kramers-Moyal sparse networks.
    \item A trained stochastic NG-RC controller generates closed-loop feedback signals to regulate high-amplitude seizure oscillations toward the resting state. Experimental results indicate effective seizure suppression, though further optimization is necessary to enhance real-time performance and clinical adaptability.
\end{itemize}

In summary, the data-driven control framework based on stochastic NG-RC offers a novel approach to orbit control of stochastic dynamical systems with multiple time scales. The method not only enables rapid and accurate trajectory tracking but also significantly reduces sample complexity compared to traditional artificial neural networks. Moreover, it inherits the advantages of reservoir computing, including structural clarity and computational efficiency.

\section{Future Directions}
\label{sec:future_work}

While this study successfully introduces a closed-loop feedback controller, several key challenges remain. Interpretability remains a critical issue for data-driven models lacking first-principles guidance. In particular, few studies have examined the controllability and control convergence of NG-RC from a stochastic dynamical systems perspective. Although we provide conditions for the controllability of controlled SDEs and establish control error convergence requirements, theoretical guidance is still lacking for selecting basis functions in the sparse networks, as well as for determining appropriate feature vectors in the stochastic NG-RC framework.

Future work will focus on two main directions:

\subsection{Enhancing Control Performance on Real Data}

In the second real-data experiment, the controller's performance in seizure suppression was not as effective as in the simulation experiments. To address this, we propose the following improvements:

\begin{enumerate}
    \item \textbf{Optimizing stochastic dynamical system modeling}. The current Kramers-Moyal network exhibits limited fitting accuracy, leading to potential model uncertainties and significant errors. Future work could integrate brain connectome information to construct a more comprehensive neuro-dynamic model, reducing reliance on treating all external inputs as noise and thereby improving modeling accuracy.
    \item \textbf{Expanding beyond Gaussian noise}. The present approach assumes Gaussian noise, yet seizure-related high-frequency oscillations may be better characterized by more complex stochastic processes such as $\alpha$-stable L\'evy processes. Future research should explore alternative noise modeling strategies to improve realism.
    \item \textbf{Analyzing and optimizing control errors}. Further investigation is needed to identify sources of control errors.
    Specifically speaking, the influence of the perturbation signal $u_t$ on control performance should be examined, with optimized perturbation designs to enhance controller stability and generalizability. Additionally, methods to mitigate trajectory deviations induced by the stochastic term $\xi_t$ should be explored to improve target trajectory tracking.
\end{enumerate}

\subsection{Extending Control to Frequency Modulation in Stochastic Oscillatory Systems}

In stochastic oscillatory systems, amplitude and frequency regulation are often coupled, meaning that modifying one parameter may influence the other. This study primarily focuses on amplitude control through a data-driven closed-loop feedback approach, without addressing frequency regulation. Future work should explore simultaneous amplitude and frequency control under the data-driven framework, aiming for a more comprehensive characterization of stochastic oscillatory dynamics. This may involve:

\begin{itemize}
    \item Introducing time-frequency analysis techniques to capture oscillatory behavior more accurately.
    \item Designing an adaptive frequency control strategy to extend the existing model, thereby enhancing robustness and accuracy in joint amplitude-frequency modulation.
\end{itemize}

In conclusion, future research will focus on improving the accuracy of stochastic NG-RC control in real-data applications and extending its capability to jointly regulate amplitude and frequency in stochastic oscillatory systems.

\section*{Acknowledgements}
Thank those who have contributed to this research.

\appendix
\section{Supplementary material}
\subsection{Stability of the Controlled SDE}
\label{Supplementary material:stability}

In the more general case, we assume that the source systems satisfy  nonautonomous $n$-dimensional stochastic differential equations 
$$
d x(t)=f(x(t), t) d t+g(x(t), t) d W(t).
$$
According to the system (\ref{sde1}),we split the control $u_t$ into a deterministic part $u_1$ and a stochastic part $u_2$. Then we have the following controlled system
\begin{equation}\label{nonautosde}
d x(t)=[f(x(t), t)+u_1]d t+[g(x(t), t)+u_2] d W(t).
\end{equation}
Given the target trajectory (or desired output) $x_{des}$. Since the real time series data $x_{des}$ will also contain both deterministic and stochastic parts, we allow the feedback control signal $u_t=(u_1,u_2)$ to counteract the effects of the deterministic and stochastic parts of $x_{des}$ on the error dynamics system, respectively.
That is, in designing the data-driven model, let $u_1 = \phi_1(x(t),t)+f_1(x_{des}), u_2 = \phi_2(x(t),t)+g_1(x_{des})$, where $f_1,g_1$ denote the deterministic and stochastic parts of $x_{des}$, respectively.
Abstractly, the external input $u_t$ is denoted as $u_t=\phi_u(x(t), t)+d{x}_{des}$ . This step is designed to be significantly different from the conventional NG-RC control model.

The dynamical system of control error $e_t = x(t)-x_{des}$ is as follows
\begin{equation}\label{e_dot}
\begin{aligned}
de_t = [f(e_t+x_{des},t) + \phi_1(e_t+x_{des},t)]dt + [g(e_t+x_{des},t)+\phi_2(e_t+x_{des},t)] d W(t).
\end{aligned}
\end{equation}
For the sake of notational aesthetics, we write $(e_t+x_{des},t)$ as $(e(t)),t)$ below.

Inspired by the work of Mao (1998)\cite{Mao1999StochasticVO}, in the following we give the conditions that are extremely necessary for the LaSalle Theorem of controlled stochastic differential equation.

(H1) Set $f,\phi_1: R^n \times R_{+} \rightarrow R^n$ and $g,\phi_2: R^n \times R_{+} \rightarrow R^{n \times m}$ are measurable functions. They satisfy the local Lipschitz condition and the linear growth condition. That is, for each $k=1,2, \ldots$, there is a $c_k>0$ such that

$$
|f(e_1, t)-f(e_2, t)| \vee|\phi_1(e_1, t)-\phi_1(e_2, t)|  \vee|g(e_1, t)-g(e_2, t)| \vee|\phi_2(e_1, t)-\phi_2(e_2, t)| \leqslant c_k|e_1-e_2|,
$$
for all $t \geqslant 0$ and those $e_1, e_2\in R^n$ with $|e_1| \vee|e_2| \leqslant k$. There is moreover a $c>0$ such that
$$
|f(e, t)| \vee|\phi_1(e, t)| \vee|g(e, t)| \vee|\phi_2(e, t)|\leqslant c(1+|e|),
$$
for all $(e, t) \in R^n \times R_{+}$.

We denote by $L^1\left(R_{+} ; R_{+}\right)$the family of all functions $\gamma: R_{+} \rightarrow R_{+}$such that $\int_0^{\infty} \gamma(t) d t<\infty$ .
Let $C^{2,1}\left(R^n \times R_{+} ; R_{+}\right)$denote the family of all nonnegative functions $V(e, t)$ on $R^n \times R_{+}$which are continuously twice differentiable in $x$ and once differentiable in $t$. Define an operator $L$ acting on $C^{2,1}\left(R^n \times R_{+} ; R_{+}\right)$ functions by
$$
\begin{aligned}
LV(e, t) &= V_t(e, t) + V_e(e, t) \left[ f(e, t) + \phi_1(e, t) \right] \\
&\quad + \frac{1}{2} \operatorname{trace}\left\{ \left[ g(e, t) + \phi_2(e, t) \right]^T V_{ee}(e, t) \left[ g(e, t) + \phi_2(e, t) \right] \right\}
\end{aligned}
$$
where
$$
\begin{gathered}
V_t(e, t)=\frac{\partial V(e, t)}{\partial t}, \quad V_e(e, t)=\left(\frac{\partial V(e, t)}{\partial x_1}, \ldots, \frac{\partial V(e, t)}{\partial x_n}\right), \\
V_{e e}(e, t)=\left(\frac{\partial^2 V(e, t)}{\partial e_i \partial e_j}\right)_{n \times n} .
\end{gathered}
$$

\begin{proposition}\label{pro1}[Stability of the Controlled SDE] 
Let (H1) hold. Assume that there is a function $V \in$ $C^{2,1}\left(R^n \times R_{+} ; R_{+}\right)$, a function $\gamma \in L^1\left(R_{+} ; R_{+}\right)$and a continuous function $w: R^n \rightarrow R_{+}$such that 
\begin{equation}\label{con1}
\lim _{|x| \rightarrow \infty} \inf _{0 \leqslant t<\infty} V(e, t)=\infty
\end{equation}
and
\begin{equation}\label{con2}
L V(e, t) \leqslant \gamma(t)-w(e), \quad(e, t) \in R^n \times R_{+} .
\end{equation}
Moreover, for all initial values $e_0 \in R^n$ there exists a constant $p>2$ such that the P-order moments of the solution of the system \eqref{e_dot} starting from the initial value $e_0$ exist
\begin{equation}\label{con3}
\sup _{0 \leqslant t<\infty} E\left|e\left(t ; e_0\right)\right|^p<\infty
\end{equation}
Then, for every $e_0 \in R^n$,
\begin{equation}\label{clu1}
\lim _{t \rightarrow \infty} V\left(e\left(t ; e_0\right), t\right) \text { exists and is finite almost surely }
\end{equation}
and, moreover,
\begin{equation}\label{clu2}
\lim _{t \rightarrow \infty} w\left(e\left(t ; e_0\right)\right)=0 \quad \text { a.s. }
\end{equation}
\end{proposition}

The condition \eqref{con1} is that the function $V(e, t)$ is radially unbounded. The conclusion \eqref{clu2} can be understood as the set $D_w= \left\{e \in R^n: w(e)=0 \right\}$ is the invariant set of the system \eqref{e_dot}.
For any initial value $e_0 \in R^n$  and sample point $\omega \in \Omega$, the conclusion \eqref{clu1}-\eqref{clu2} hold.
By \eqref{con1}, $\left\{e\left(t, \omega ; e_0\right): t \geqslant 0\right\}$ must be bounded so there is a subsequence $\left\{e\left(t_k, \omega ; e_0\right): k \geqslant 1\right\}$ which converges to some $\hat{e} \in R^n$. Since $w$ is continuous, we have $w(\hat{e})=0$, i.e., $\hat{e} \in D_w$ whence $D_w \neq \varnothing$. Let $d\left(e, D_w\right)$ denote the distance between $e$ and set $D_w$, that is $d\left(e, D_w\right)=\min _{\bar{e} \in D_w}|e-\bar{e}|$. Then \eqref{clu2} means
$$
\lim _{t \rightarrow \infty} d\left(x\left(t ; x_0\right), D_w\right)=0 \quad \text { a.s. }
$$
In other words, all the solutions of Eq.\eqref{e_dot} will asymptotically approach $D_w$ with probability one. 
Before proving the proposition \ref{pro1}, we give the following facts and verify a lemma.

\begin{theorem}[Kolmogorov-Centsov theorem\cite{karatzas1991brownian}]\label{fact1}

Suppose that the n-dimensional stochastic process $X(t)$ on $t \geqslant 0$ satisfies the following conditions
$$
E|X(t)-X(s)|^\alpha \leqslant C|t-s|^{1+\beta}, \quad 0 \leqslant s, t<\infty,
$$
holds for some positive constants $\alpha, \beta$ and $C$. Then there exists a continuous modification $\widetilde{X}(t)$ to $X(t)$ which has the property that for the exponent $\gamma \in(0, \beta / \alpha)$, there exists a positive random variable $h(\omega)$ such that
$$
P\left\{\omega: \sup _{\substack{0<t-s<h(\omega) \\ 0 \leqslant s, t<\infty}} \frac{|\tilde{X}(t, \omega)-\tilde{X}(s, \omega)|}{|t-s|^\gamma} \leqslant \frac{2}{1-2^{-\gamma}}\right\}=1 .
$$
In other words, almost every sample path of $\widetilde{X}(t)$ is locally but uniformly Hölder-continuous with index $\gamma$.

\end{theorem}

\begin{lemma}\label{lemma}
Let (H1) and condition \eqref{con3} hold. Set
$$
y(t):=\int_0^t \left[g(e(s), s)+\phi_2(e(s), s)\right] d W(s) \quad  \quad t \geqslant 0.
$$
Here we write $e\left(t ; e_0\right)$ for $e(t)$. Then $y(t)$ almost every sample path is uniformly continuous on $t \geqslant 0$.
\end{lemma}
\begin{proof}
By the moment inequality for stochastic integrals \cite{mao2007stochastic}, for $0 \leqslant s<t<\infty$ and $p>2$,
$$
E|y(t)-y(s)|^p \leqslant\left[\frac{p(p-1)}{2}\right]^{p / 2}(t-s)^{(p-2) / 2} \int_s^t E|g(e(r), r)+\phi_2(e(r), r)|^p d r.
$$
But by (H1) and \eqref{con3} we can derive that
$$
E|g(e(r), r)+\phi_2(e(r), r)|^p \leqslant E[2c(1+|e(r)|)]^p \leqslant(2 c)^p 2^{p-1}\left(1+E|e(r)|^p\right) \leqslant 2^{2p-1} c^p(1+\bar{K}),
$$
where $\bar{K}:=\sup _{0 \leqslant t<\infty} E|e(t)|^p<\infty$. Therefore
$$
E|y(t)-y(s)|^p \leqslant\left[\frac{p(p-1)}{2}\right]^{p / 2}2^{2p-1} c^p(1+\bar{K})(t-s)^{1+(p-2) / 2}.
$$
Considering that $y(t)$ is continuous, it follows from Theorem \ref{fact1} that ${y}(t)$ is locally but uniformly Hölder continuous for almost every sample path. And for every index $\gamma \in(0,(p-2) / 2 p)$, $y(t)$ almost every sample path must be uniformly continuous.
\end{proof}

\begin{lemma}\label{fact2}
\cite{liptser2012theory}

Let $A(t)$ and $U(t)$ be two continuously adapted increasing processes defined on $t \geqslant 0$ with $A(0)=U(0)=0$. Let $M(t)$ be a real-valued continuous local martingale with $M(0)=0$. Let $\xi$ be a non-negative and $\mathscr{F}_0$-measurable random variable. Define
$$
X(t)=\xi+A(t)-U(t)+M(t) \quad \quad t \geqslant 0.
$$
If $X(t)$ is non-negative, then 
$$
\left\{\lim _{t \rightarrow \infty} A(t)<\infty\right\} \subset\left\{\lim _{t \rightarrow \infty} X(t) \text { exists and is finite }\right\} \cap\left\{\lim _{t \rightarrow \infty} U(t)<\infty\right\} \text { a.s. }
$$
In particular, if $\lim _{t \rightarrow \infty} A(t)<\infty$ a.s., then for almost all $\omega \in \Omega$
$$\lim _{t \rightarrow \infty} X(t, \omega)  \text { exists and is finite, and  }\lim _{t \rightarrow \infty} U(t, \omega)<\infty$$.
\end{lemma}

The following is based on the above theorem\ref{fact1}, lemma \ref{lemma}-\ref{fact2}, and condition\eqref{con1}-\eqref{con3}to prove the proposition\ref{pro1}.
\begin{proof}
Fix any initial value $e_0$ and for simplicity write $e\left(t ; e_0\right)=e(t)$. By $It\hat{o}^{'}s$ formula and condition\eqref{con2}
\begin{equation*}
\begin{aligned}
V(e(t), t)= & V\left(e_0, 0\right) + \int_0^t\left\{ V_t(e,s)+V_e(e,s)\left[f(e,s)+\phi_1(e,s)\right]\right\}ds\\
+&\int_0^t\frac{1}{2} \operatorname{trace}\left([g(e, s)+\phi_2(e,s)]^T V_{x x}(e, s) [g(e, s)+\phi_2(e,s)]\right) d s\\
& +\int_0^t V_e(e, s) [g(e, s)+\phi_2(e,s)] d W(s)\\
=& V\left(e_0, 0\right) + \int_0^t L V(e(s), s) d s +\int_0^t V_e(e(s), s) [g(e(s), s)+\phi_2(e(s),s)] d W(s) \\
= & V\left(e_0, 0\right)+\int_0^t \gamma(s) d s-\int_0^t[\gamma(s)-L V(e(s), s)] d s \\
+&\int_0^t V_e(e(s), s) [g(e(s), s)+\phi_2(e(s),s)] d W(s) \\
\leqslant & V\left(e_0, 0\right)+\int_0^t \gamma(s) d s-\int_0^t w(e(s)) d s +\int_0^t V_e(e(s), s) [g(e(s), s)+\phi_2(e(s),s)] d W(s) .
\end{aligned}
\end{equation*}
Since $\int_0^{\infty} \gamma(s) d s<\infty$ and the condition \eqref{con2}, we have , by Lemma \ref{fact2}, that for almost all $\omega \in \Omega$
\begin{equation}\label{fact}
\int_0^{\infty} w(e(t, \omega)) d t<\infty,
\end{equation}
and
$$
\lim _{t \rightarrow \infty} V\left(e\left(t ; e_0\right), t\right) \text { exists and is finite.}
$$
Conclusion \eqref{clu1} obtained.

For almost surely (a.s.) in the second conclusion, we first need to construct a collection of samples such that the conclusion \eqref{clu2} holds exactly on these samples.
Under the condition \eqref{con3}, let the
$$
y(t):=\int_0^t [g(e(s), s)+\phi_2(e(s), s)] d W(s) \quad \text { on } \quad t \geqslant 0.
$$
By Lemma \ref{lemma}, it follows that almost every sample path of $y(t)$ on $t \geqslant 0$ is uniformly continuous.
Based on the facts \eqref{fact} and conclusions \eqref{clu1} deduced above, we can choose a subset $\bar{\Omega} \subseteq \Omega$ that satisfies $P(\bar{\Omega})=1$ such that, for each sample $\omega \in \bar{\Omega}$, the facts\eqref{fact} and conclusion \eqref{clu1} both hold, and $y(t, \omega)$ is uniformly continuous on $t \geqslant 0$. Fix any $\omega$ in the set $\bar{\Omega}$. Based on the conclusion \eqref{clu1}, we have
$$
\sup _{0 \leqslant t<\infty} V(e(t, \omega), t)<\infty.
$$
Thus, under condition \eqref{con1}, there exists a positive random variable $h(\omega)$ such that
$$
|e(t, \omega)| \leqslant h(\omega) \quad \quad t \geqslant 0.
$$
Set
$$
z(t)=\int_0^t [f(e(s), s)+\phi_1(e(s),s)] d s.
$$
From Assumption (H1) we can calculate that for any $0 \leqslant s<t<\infty$) there are
$$
\begin{aligned}
|z(t, \omega)-z(s, \omega)| & \leqslant \int_s^t[|f(e(r, \omega), r)|+|\phi_1(e(r, \omega), r)|] d r \\
& \leqslant 2c \int_s^t(1+|e(r, \omega)|) d r \leqslant 2c(1+h(\omega))(t-s),
\end{aligned}
$$
This shows that $z(t, \omega)$ is uniformly continuous on $t \geqslant 0$.
According to the formula \eqref{e_dot}, we have $e(t)=e_0+z(t)+y(t)$.
Since $\omega \in \bar{\Omega}$ is arbitrarily chosen, then for any $\omega \in \bar{\Omega}, e(t, \omega)$ is uniformly continuous on $t \geqslant 0$.

It follows that the conclusion \eqref{clu2} holds on the set $\bar{\Omega}$, i.e.
$$
\lim _{t \rightarrow \infty} w(e(t, \omega))=0 \quad \quad \omega \in \bar{\Omega} .
$$
Using the counterfactual, suppose there exists $\hat{\omega} \in \bar{\Omega}$ such that
$$
\limsup _{t \rightarrow \infty} w(e(t, \hat{\omega}))>0.
$$
Then there exists $\varepsilon>0$ and a positive increasing sequence $\left\{t_k\right\}_{k \geqslant 1}$ ($t_k+1<t_{k+1}$) such that
\begin{equation}\label{9}
w\left(e\left(t_k, \hat{\omega}\right)\right)>\varepsilon \quad \quad k \geqslant 1.
\end{equation}
Set $\bar{S}_h=\left\{e \in R^n:|e| \leqslant h(\hat{\omega})\right\}$, where the random variable $h(\hat{\omega})$ has been defined above, then there is $\{e(t, \hat{\omega}): t \geqslant 0\} \subset \bar{S}_h$. Since the function $w(\cdot)$ is continuous, it must be uniformly continuous on $\bar{S}_h$, then there exists $\delta_1>0$ such that
$$
|w(e_1)-w(e_2)|<\frac{\varepsilon}{2} \quad \text {, if } \quad e_1, e_2 \in \bar{S}_h, \quad|e_1-e_2|<\delta_1.
$$
On the other hand, $e(t, \hat{\omega})$ is uniformly continuous on $t \geqslant 0$, then we can find the constant $\delta_2 \in(0,1)$ such that
$$
|e(t, \hat{\omega})-e(s, \hat{\omega})|<\delta_1 \quad \text {, if  } \quad 0 \leqslant t, s<\infty, \quad|t-s| \leqslant \delta_2.
$$
Thus, we can get that for any $n \geqslant 1$
$$
\left|w\left(e\left(t_k, \hat{\omega}\right)\right)-w(e(t, \hat{\omega}))\right|<\frac{\varepsilon}{2} \quad \text {, if } \quad t_k \leqslant t \leqslant t_k+\delta_2 .
$$
Combined with Eq. \eqref{9}, we get
$$
w(e(t, \hat{\omega})) \geqslant w\left(e\left(t_k, \hat{\omega}\right)\right)-\left|w\left(e\left(t_k, \hat{\omega}\right)\right)-w(e(t, \hat{\omega}))\right|>\varepsilon-\frac{\varepsilon}{2}=\frac{\varepsilon}{2} .
$$
Therefore
$$
\int_0^{\infty} w(e(t, \hat{\omega})) d t \geqslant \sum_{k=1}^{\infty} \int_{t_k}^{t_k+\delta_2} w(e(t, \hat{\omega})) d t \geqslant \sum_{k=1}^{\infty} \frac{\varepsilon \delta_2}{2}=\infty.
$$
However, this result is completely contrary to the conclusion \eqref{clu1}. Therefore, it can be concluded that \eqref{clu2} holds for all $w \in \bar{\Omega}$, i.e., the conclusion \eqref{clu2} holds almost surely. The proposition is proved.
\end{proof}

It is clear from the above proof that the condition \eqref{con3} is only used to prove that the
$$
y(t):=\int_0^t [g(e(s), s)+\phi_2(e(s), s)] d W(s),
$$
almost every sample path on $t \geqslant 0$ is uniformly continuous. Therefore, this condition can be replaced by other conditions. One of the most commonly used alternative conditions is that the diffusion term $g(e,t)$ is bounded. Furthermore, going further and investigating the extension of Lassalle's theorem to controlled SDDEs one can draw on the work of Mao (2002)\cite{Mao2002ANO}.
\appendix
\subsection{Figure Supplement}
\begin{figure}[htbp]
    \centering
    \includegraphics[width=\textwidth]{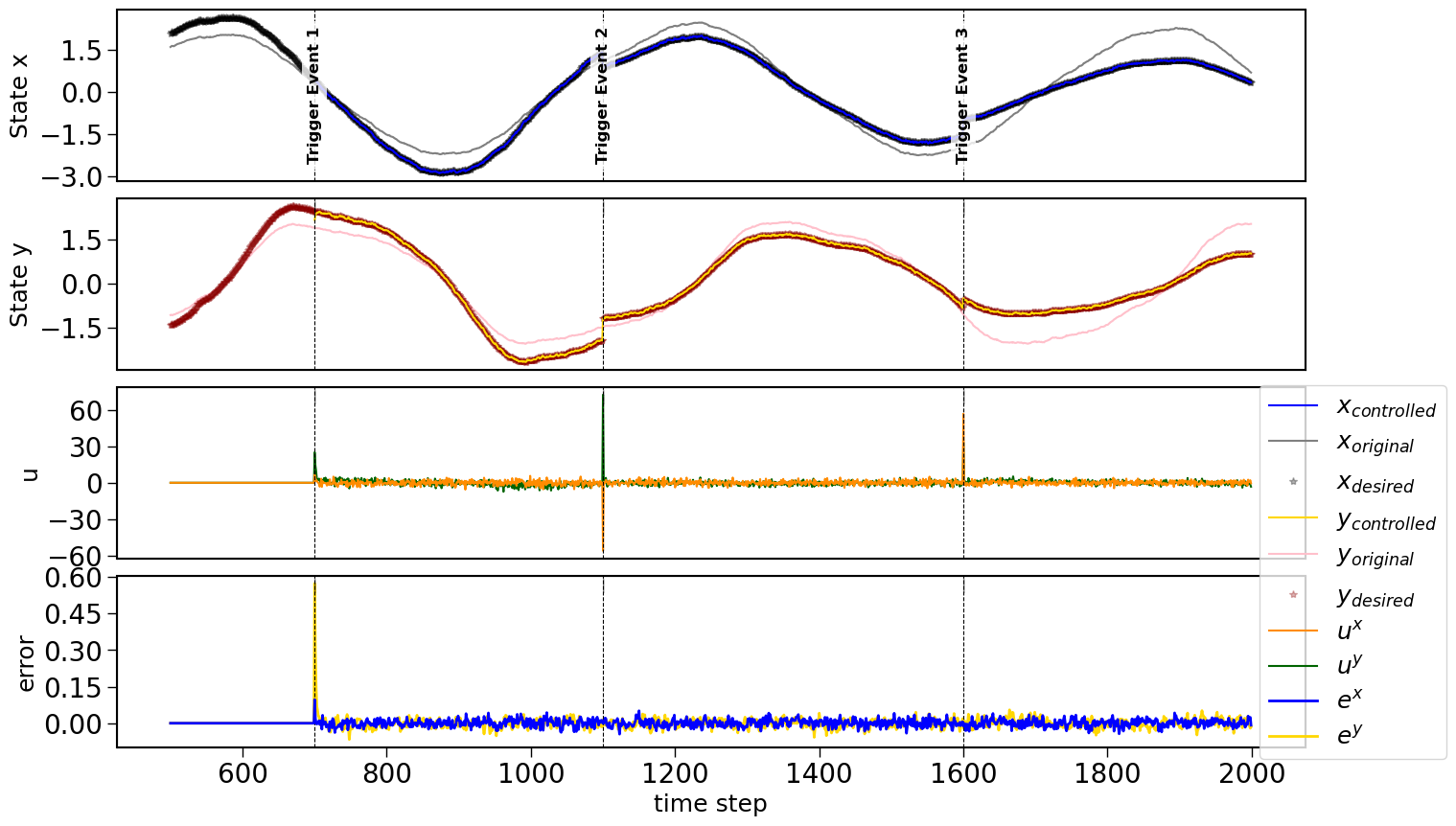}
    \caption{Control performance of the stochastic NG-RC under low-intensity noise without multiple time scales. \\
    a) The three vertical black dotted lines mark key points: (1) the time when control is started, (2) the first transition in the desired trajectory, and (3) the second transition in the desired trajectory. \\
    b) The top row shows the control results for the $x$-variable under the stochastic NG-RC. The grey curve represents the original trajectory of the Van der Pol (VDP) system. The black pentagram indicates the desired trajectory: before the second dashed line, the trajectory is an outward expansion ( $1.2 \times$ amplitude); after the second dashed line, it contracts to $0.8 \times$ the original amplitude, and further contracts to $0.5 \times$ after the third dashed line. The blue curves indicate the tracking results of the stochastic NG-RC controller, wihch start from the first dashed line. \\
    c) The second row shows the control results for the $y$-variable under the stochastic NG-RC. The pink curve represents the original trajectory, the red pentagram marks the desired trajectory, and the gold curve shows the controller's tracking response. \\
    d) The third row displays the feedback control signals $u^x$ variable (in orange) and $u^y$ variable (in green), which are computed inversely by the stochastic NG-RC controller based on the desired trajectory. \\
    e) The bottom row presents the control error $e_i \in \mathbb{R}^2$ at each step, reflecting the difference between the system state and the desired trajectory.}
    \label{fig:2d_add control}
\end{figure}

\bibliographystyle{unsrt}
\bibliography{main}

\end{document}